\documentclass[12pt,a4paper,reqno, oneside]{amsart}%twoside]{amsart}
 
\usepackage[utf8]{inputenc}
\usepackage[T1]{fontenc}
\usepackage{lmodern}
\usepackage{amsmath}
\usepackage{amssymb}
\usepackage{amsthm}
\allowdisplaybreaks
\usepackage{geometry}
\usepackage{graphicx}
\usepackage{bbm}
\usepackage{float}
\usepackage{enumerate}
\usepackage{scrhack}
\usepackage{hyperref}
\usepackage{cleveref}

\usepackage{fancyhdr}
\newtheoremstyle{theorem}	% name
   {}       			% Space above
   {}      				% Space below
   {\itshape}  				% Body font
   {\parindent} 			% Indent amount (empty = no indent, \parindent = para indent)
   {\bfseries} 				% Thm head font
   {}         				% Punctuation after thm head
   {.7em}      				% Space after thm head: " " = normal interword space;
   {}          				% Thm head spec (can be left empty, meaning 'normal')

\newtheoremstyle{definition}% name
  {}       				% Space above
  {}      				% Space below
  {} 		 				% Body font
  {\parindent} 				% Indent amount (empty = no indent, \parindent = para indent)
  {\bfseries} 				% Thm head font
  {}         				% Punctuation after thm head
  {.7em}      				% Space after thm head: " " = normal interword space;
  {}          				% Thm head spec (can be left empty, meaning 'normal')
  
\newtheoremstyle{remark}% name
  {}       			% Space above
  {}      			% Space below
  {} 					% Body font
  {\parindent} 			% Indent amount (empty = no indent, \parindent = para indent)
  {\itshape} 			% Thm head font
  {.}         			% Punctuation after thm head
  {.7em}      			% Space after thm head: " " = normal interword space;
  {}          			% Thm head spec (can be left empty, meaning 'normal')

\theoremstyle{theorem}
\newtheorem{theorem}{Theorem}[section]
\newtheorem{corollary}[theorem]{Corollary}
\newtheorem{lemma}[theorem]{Lemma}
\newtheorem{proposition}[theorem]{Proposition}

\theoremstyle{remark}
\newtheorem{remark}[theorem]{Remark}

\theoremstyle{definition}

\newtheorem{example}[theorem]{Example}

\newcommand{\Ebb}{\mathbb{E}}
\newcommand{\Fbb}{\mathbb{F}}

\newcommand{\Nbb}{\mathbb{N}}
\newcommand{\Pbb}{\mathbb{P}}
\newcommand{\Qbb}{\mathbb{Q}}
\newcommand{\Rbb}{\mathbb{R}}

\newcommand{\Bcal}{\mathcal{B}}
\newcommand{\Ccal}{\mathcal{C}}
\newcommand{\Ecal}{\mathcal{E}}
\newcommand{\Fcal}{\mathcal{F}}

\newcommand{\Kcal}{\mathcal{K}}
\newcommand{\Lcal}{\mathcal{L}}

\newcommand{\Pcal}{\mathcal{P}}

\newcommand{\Wcal}{\mathcal{W}}
\newcommand{\Xcal}{\mathcal{X}}

\newcommand{\bfr}{\mathfrak{b}}
\newcommand{\Cfrak}{\mathfrak{C}}

\newcommand{\Lip}{\text{Lip}}

\DeclareMathOperator*{\esssup}{ess\,sup}

\newcommand{\btilde}{\tilde{b}}
\newcommand{\bhat}{\hat{b}}
\newcommand{\bbar}{\overline{b}}
\newcommand{\dd}{\text{d}}

\newcommand{\Ep}[2]{\Ebb\negthinspace\left[\left\Vert #1 \right\Vert^{#2} \right]^{\frac{1}{#2}}\negthinspace}
\newcommand{\EPpO}[3]{\Ebb_{#2}\negthinspace\left[\left\Vert #1 \right\Vert^{#3} \right]\negthinspace}
\newcommand{\Eabs}[1]{\Ebb\negthinspace\left[\left\Vert #1 \right\Vert \right]\negthinspace}
\newcommand{\EW}[1]{\Ebb\negthinspace\left[ #1 \right]\negthinspace}

\newcommand{\EpE}[2]{\Ebb\negthinspace\left[\left\vert #1 \right\vert^{#2} \right]^{\frac{1}{#2}}\negthinspace}

\usepackage{color}
\definecolor{darkgreen}{rgb}{0, .5, 0}
\definecolor{darkred}{rgb}{.5, 0, 0}

\begin{document}

%Title________________________________________________________________

\title[Mean-field SDEs with irregular expectation functional]{Strong Solutions of Mean-Field SDEs with irregular expectation functional in the drift}
\author[M.Bauer]{Martin Bauer}
\address{M. Bauer: Department of Mathematics, LMU, Theresienstr. 39, D-80333 Munich, Germany.}
\email{bauer@math.lmu.de}
\author[T. Meyer-Brandis]{Thilo Meyer-Brandis}
\address{T. Meyer-Brandis: Department of Mathematics, LMU, Theresienstr. 39, D-80333 Munich, Germany.}
\email{meyerbra@math.lmu.de}
\date{\today}
\maketitle

%(Abstract)___________________________________________________________
\begin{center}
\parbox{13cm}{
\begin{footnotesize}
\textbf{\textsc{Abstract.}} We analyze multi-dimensional mean-field stochastic differential equations where the drift depends on the law in form of a Lebesgue integral with respect to the pushforward measure of the solution. We show existence and uniqueness of Malliavin differentiable strong solutions for irregular drift coefficients, which in particular include the case where the drift depends on the cumulative distribution function of the solution. Moreover, we examine the solution as a function in its initial condition and introduce sufficient conditions on the drift to guarantee differentiability.
%to which extend the coefficients have to be regularized opposed to the results in \cite{Bauer_StrongSolutionsOfMFSDEs} in order to get strong differentiability in the initial condition of the solution. 
Under these assumptions we then show that the Bismut-Elworthy-Li formula proposed in \cite{Bauer_StrongSolutionsOfMFSDEs} %and show that the probabilistic representation 
holds in a strong sense, i.e. we give a probabilistic representation of the strong derivative with respect to the initial condition of expectation functionals of strong solutions to our type of mean-field equations in one-dimension. \\[0.2cm]
\end{footnotesize}
}
\end{center}

%(Main part)%%%%%%%%%%%%%%%%%%%%%%%%%%%%%%%%%%%%%%%%%%%%%%%%%%%%%%%%%%%%%%%%%%%%%%%%%%%%
\section{Introduction}
Throughout this paper, let $T>0$ be a given time horizon. As an extension of stochastic differential equations, mean-field stochastic differential equations (hereafter mean-field SDEs), also referred to as McKean-Vlasov equations, given by
\begin{align}\label{eq:RegMainMcKeanVlasov}
	dX_t^x = \bbar\left(t,X_t^x,\Pbb_{X_t^x}\right) dt + \overline{\sigma}\left(t,X_t^x, \Pbb_{X_t^x} \right) dB_t,~ t\in [0,T], ~ X_0^x = x \in \Rbb^d,
\end{align}
allow the coefficients to depend on the law of the solution in addition to the solution process. Here, $\bbar: [0,T] \times \Rbb^d \times \Pcal_1(\Rbb^d) \to \Rbb^d$ and $\overline{\sigma}: [0,T] \times \Rbb^d \times \Pcal_1(\Rbb^d) \to \Rbb^{d\times n}$ are some given drift and volatility coefficients, $(B_t)_{t\in[0,T]}$ is $n$-dimensional Brownian motion, 
\begin{align*}
		\Pcal_1(\Rbb^d) := \left\lbrace \mu \left| \mu \text{ probability measure on } (\Rbb^d, \Bcal(\Rbb^d)) \text{ with } \int_{\Rbb^d} \Vert x \Vert d\mu(x) < \infty \right.\right\rbrace
\end{align*}
is the space of probability measures over $(\Rbb^d,\Bcal(\Rbb^d))$ with existing first moment, and $\Pbb_{X_t^x}$ is the law of $X_t^x$ with respect to the underlying probability measure $\Pbb$. \par

Mean-field SDEs arised from Boltzmann's equation in physics, which is used to model weak interaction between particles in a multi-particle system, and were first studied by Vlasov \cite{Vlasov_VibrationalPropertiesofElectronGas}, Kac \cite{Kac_FoundationsOfKineticTheory} and McKean \cite{McKean_AClassOfMarkovProcess}. Nowadays the study of mean-field SDEs is an active research field with numerous applications. Various extensions such as replacing the driving noise by a Lévy process or considering backward equations have been examined e.g.~in \cite{BuckdahnDjehicheLiPeng_MFBSDELimitApproach}, \cite{BuckdahnLiPeng_MFBSDEandRelatedPDEs}, and \cite{JourdainMeleardWojbor_NonlinearSDEs}. A cornerstone in the application of mean-field SDEs in Economics and Finance was set by Lasry and Lions with their work on mean-field games in \cite{LasryLions_MeanFieldGames}, see also \cite{Cardaliaguet_NotesOnMeanFieldGames} for a readily accessible summary of Lions' lectures at Collège de France. Carmona and Delarue developed a probabilistic approach to mean-field games opposed to the analytic one taken in \cite{LasryLions_MeanFieldGames}, see e.g. \cite{CarmonaDelarue_ProbabilisticAnalysisofMFG}, \cite{CarmonaDelarue_MasterEquation}, \cite{CarmonaDelarue_FBSDEsandControlledMKV}, \cite{CarmonaDelarueLachapelle_ControlofMKVvsMFG}, and \cite{CarmonaLacker_ProbabilisticWeakFormulationofMFGandApplications} as well as the monographs \cite{CarmonaDelarue_Book}. A more recent application of the concept of mean-fields is in the modeling of systemic risk, in particular in models for inter-bank lending and borrowing, see e.g. \cite{CarmonaFouqueMousaviSun_SystemicRiskandStochasticGameswithDelay}, \cite{CarmonaFouqueSun_MFGandSystemicRisk}, \cite{FouqueIchiba_StabilityinaModelofInterbankLending}, \cite{FouqueSun_SystemicRiskIllustrated}, \cite{GarnierPapanicolaouYang_LargeDeviationsforMFModelofSystemicRisk}, \cite{KleyKlueppelbergReichel_SystemicRiskTroughContagioninaCorePeripheryStructuredBankingNetwork}, and the cited sources therein. \par 
	
In this paper we analyze (strong) solutions of multi-dimensional mean-field SDEs of the form
\begin{align}\label{lebesgueMFSDE}
  dX_t^x = b\left(t,X_t^x,\int_{\Rbb^d} \varphi\left(t,X_t^x,z\right) \Pbb_{X_t^x}(dz) \right) dt + dB_t,~ t\in[0,T],~ X_0^x=x \in \Rbb^d,
\end{align}
for $b, \varphi: [0,T] \times \Rbb^d \times \Rbb^d \to \Rbb^d$. This mean-field SDEs generalize two commonly used models in the literature, where the first one considers $b(t,y,z) = z$, see e.g. \cite{MishuraVeretennikov_SolutionsOfMKV}, or \cite{BuckdahnDjehicheLiPeng_MFBSDELimitApproach} where the authors consider backward mean-field SDEs, and in the second model $\varphi(t,y,z) = \overline{\varphi}(z)$ for some $\overline{\varphi}: \Rbb^d \to \Rbb^d$, see e.g. \cite{Banos_Bismut}. Note that putting $\overline{\sigma} \equiv 1$ and 
\begin{align}\label{eq:MFDrift}
  \bbar(t,y,\mu) = (b \diamond \varphi)(t,y,\mu) := b\left(t,y,\int_{\Rbb^d} \varphi(t,y,z) \mu(dz)\right),
\end{align} 
yields that mean-field SDE \eqref{lebesgueMFSDE} is recognized as a special case of the general mean-field SDE \eqref{eq:RegMainMcKeanVlasov}. 

The first main contribution of this paper is to establish existence and uniqueness of weak and strong solutions of mean-field SDE \eqref{lebesgueMFSDE} with irregular drift. Further, we show that the strong solutions are Malliavin differentiable. For coefficients $\bbar$ and $\overline{\sigma}$ in the general mean-field SDE\eqref{eq:RegMainMcKeanVlasov} fulfilling typical regularity assumptions such as linear growth and Lipschitz continuity, existence and uniqueness is well-studied, see e.g \cite{BuckdahnLiPengRainer_MFSDEandAssociatedPDE}. In \cite{Chiang_MKVWithDiscontinuousCoefficients} the existence of strong solutions is shown for time-homogeneous mean-field SDEs \eqref{eq:RegMainMcKeanVlasov} with drift coefficients $\bbar$ that are of linear growth and allow for certain discontinuities in the space variable $y$ and are Lipschitz in the law variable $\mu$. In the time-inhomogeneous case it is shown in \cite{MishuraVeretennikov_SolutionsOfMKV} that there exists a strong solution of mean-field SDE \eqref{lebesgueMFSDE} in the special case $b(t,y,z) = z$ under the assumption that $\varphi$ is of linear growth. The special case of mean-field SDE \eqref{lebesgueMFSDE}, where $\varphi(t,y,z) = \overline{\varphi}(z)$, is treated in \cite{de2015strong}. Here the author assumes that the drift coefficient $b$ is bounded and continuously differentiable in the law variable $z$ and $\overline{\varphi}$ is $\alpha$-Hölder continuous for some $0< \alpha \leq 1$. The work that is the closest to our analysis presented in the following are \cite{Bauer_MultiDim} and \cite{Bauer_StrongSolutionsOfMFSDEs}, where for additive noise, i.e.~$\overline{\sigma} \equiv 1$, existence and uniqueness of weak and Malliavin differentiable strong solutions of mean-field SDE \eqref{eq:RegMainMcKeanVlasov} is shown for irregular drift coefficients $\bbar$ including the case of bounded coefficients $\bbar$ that are allowed to be merely measurable in the space variable $y$ and continuous in the law variable $\mu$.\par 
	Considering mean-field SDE \eqref{lebesgueMFSDE}, first existence and uniqueness results of solutions for irregular drifts are inherited from results in \cite{Bauer_MultiDim} on the general mean-field SDE \eqref{eq:RegMainMcKeanVlasov} by specifying $b$ and $\varphi$ such that $\bbar$ in \eqref{eq:MFDrift} fulfills the assumptions in \cite{Bauer_MultiDim}. We derive these conditions in \Cref{sec:Recall}. However, in order to guarantee continuity in the law variable $\mu$ required in \cite{Bauer_MultiDim} we cannot allow for irregular $\varphi$, in particular we need that $\varphi$ is Lipschitz continuous in the third variable. This excludes interesting examples where $\varphi$ is irregular, as for example the case when $\varphi(t,x,z) = \mathbbm{1}_{\lbrace z\leq u \rbrace}$, $u\in \Rbb$, and thus the case where the drift $b\left(t,X_t^x,F_{X_t^x}(u)\right)$ depends on the distribution function $F_{X_t^x}(\cdot)$ of the solution is not covered. The objective of this paper is thus to show existence and uniqueness of weak and Malliavin differentiable strong solutions of mean-field SDE \eqref{lebesgueMFSDE} where we relax the conditions on $\varphi$ even further and merely assume that $\varphi$ is measurable and of at most linear growth. The assumptions on the drift function $b$ are inherited from \cite{Bauer_MultiDim} which includes the case of merely measurable coefficients of at most linear growth that are continuous in the law variable $z$. As one application we obtain a global version of Carathéodory's existence theorem for ODEs.

%%%%%%%%%%%%%%%%%%%%%%%%%%%%%%%%%%%%%%%%%%%%%%%%%%%%%%%%%%%%%%%%%%%%%%%%%%%%%%%%%%%
	In the second part of the paper the main objective is to study the differentiability in the initial condition $x$ of the expectation functional $\Ebb[\Phi(X_T^x)]$ and to give a Bismut-Elworthy-Li type representation of $\partial_x \Ebb[\Phi(X_T^x)]$\footnote{Here, $\partial_x$ denotes the Jacobian with respect to the variable $x$.}, where $\Phi: \Rbb \to \Rbb$ and $(X_t^x)_{t\in [0,T]}$ is the unique strong solution of the one-dimensional mean-field SDE \eqref{lebesgueMFSDE}, i.e. $d=1$. In \cite{Bauer_StrongSolutionsOfMFSDEs} it is shown that $\Ebb[\Phi(X_T^x)]$ is Sobolev differentiable in its initial condition for a broad range of irregular drift coefficients and for $\Phi$ fulfilling merely some integrability condition, and a Bismut-Elworthy-Li formula is derived. However, for various purposes it is of interest to understand when the derivative $\partial_x \Ebb[\Phi(X_T^x)]$ exists in a strong sense. For example, the weak derivative does not allow for a satisfactory interpretation of $\partial_x \Ebb[\Phi(X_T^x)]$ as a sensitivity measure in the sense of the so-called {\it Delta} from Mathematical Finance.
For the case $\varphi(t,y,z) = \overline{\varphi}(z)$ and for smooth coefficients, \cite{Banos_Bismut} provides a Bismut-Elworthy-Li formula for the continuous derivative $\partial_x \Ebb[\Phi(X_T^x)]$. We here show that $\Ebb[\Phi(X_T^x)]$ is continuously differentiable for a large family of irregular drift coefficients. More precisely, we require $b$ and $\varphi$ in addition to the assumptions for existence and uniqueness of strong solutions to be sufficiently regular in the law variable $z$. For these coefficients the Bismut-Elworthy-Li representation from \cite{Bauer_StrongSolutionsOfMFSDEs} thus holds in a strong sense. As a first step to obtain this result, we also need to study strong differentiability of $X^x$ in its initial condition $x$. In particular, we show that if $b$ and $\varphi$ are continuously differentiable in the space variable $y$ and the law variable $z$ then $X_t^x$ is continuously differentiable in $x$.

%Structure%%%%%%%%%%%%%%%%%%%%%%%%%%%%%%%%%%%%%%%%%%%%%%%%%%%%%%%%%%%%%%%%%%%%%%%%%%
The paper is structured as follows. In Section~\ref{sec:Recall} we recall the results from \cite{Bauer_MultiDim} and \cite{Bauer_StrongSolutionsOfMFSDEs} and apply it to the case of mean-field SDEs of type \eqref{lebesgueMFSDE}. These results will be employed in the remaining parts of the paper. In Section~\ref{sec:Solution} we weaken the assumptions on $\varphi$ and show existence, uniqueness, and Malliavin differentiability of solutions of mean-field equation \eqref{lebesgueMFSDE}. Finally, Section~\ref{sec:Regularity} deals with the first variation process $(\partial_x X_t^x )_{t\in[0,T]}$ and provides a Bismut-Elworthy-Li formula for the continuous derivative $\partial_x \Ebb[\Phi(X_T^x)]$ for irregular drift coefficients in the one-dimensional case.

%Notation%%%%%%%%%%%%%%%%%%%%%%%%%%%%%%%%%%%%%%%%%%%%%%%%%%%%%%%%%%%%%%%%%%%%%%%%%%%%%%%
\vspace{1cm}
\textbf{Notation:} Subsequently we list some of the most frequently used notations. For this, let $(\mathcal{X},d_{\mathcal{X}})$ and $(\mathcal{Y},d_{\mathcal{Y}})$ be two metric spaces.
\begin{itemize}
\item By $\Vert \cdot \Vert$ we denote the euclidean norm.
\item $\Ccal(\mathcal{X};\mathcal{Y})$ denotes the space of continuous functions $f:\mathcal{X} \to \mathcal{Y}$. If $\mathcal{X} = \mathcal{Y}$ we write $\Ccal(\mathcal{X}) := \Ccal(\mathcal{X}; \mathcal{X})$.
\item $\Ccal_0^{\infty}(\Xcal)$ denotes the space of smooth functions $f: \Xcal \to \Rbb$ with compact support.
\item For every $C>0$ we define the space $\Lip_C(\mathcal{X};\mathcal{Y})$ of functions $f:\mathcal{X}\to \mathcal{Y}$ such that
\begin{align*}
	d_{\mathcal{Y}}(f(x_1),f(x_2)) \leq C d_{\mathcal{X}}(x_1,x_2), \quad \forall x_1,x_2 \in \mathcal{X}
\end{align*}
as the space of Lipschitz functions with Lipschitz constant $C>0$. Furthermore, we define $\Lip(\mathcal{X};\mathcal{Y}) := \bigcup_{C>0} \Lip_C(\mathcal{X};\mathcal{Y})$ and denote by $\Lip_C(\mathcal{X}):= \Lip_C(\mathcal{X};\mathcal{X})$ and $\Lip(\mathcal{X}) := \Lip(\mathcal{X};\mathcal{X})$, respectively, the space of Lipschitz functions mapping from $\mathcal{X}$ to $\mathcal{X}$.
\item $\Ccal^{1,1}_{b,C}(\Rbb^d)$ denotes the space of continuously differentiable functions $f: \Rbb^d \to \Rbb^d$ such that there exists a constant $C>0$ with
	\begin{enumerate}[(a)]
		\item $\sup_{y\in \Rbb^d} \Vert f'(y) \Vert \leq C$, and
		\item $(y \mapsto f'(y)) \in \Lip_C(\Rbb^d)$.
	\end{enumerate}
	Here $f'$ denotes the Jacobian of $f$. We define $\Ccal^{1,1}_b(\Rbb^d) := \bigcup_{C>0} \Ccal^{1,1}_{b,C}(\Rbb^d)$.
\item $\Cfrak([0,T] \times \Rbb^d \times \Rbb^d)$ is the space of functions $f:[0,T] \times \Rbb^d \times \Rbb^d \to \Rbb^d$ such that there exists a constant $C>0$ with
	\begin{enumerate}[(a)]
		\item $(y \mapsto f(t,y,z)) \in \Ccal_{b,C}^{1,1}(\Rbb^d)$ for all $t\in[0,T]$ and $z \in \Rbb^d$, and
		\item $(z \mapsto f(t,y,z)) \in \Ccal_{b,C}^{1,1}(\Rbb^d)$ for all $t\in[0,T]$ and $y\in \Rbb^d$.
	\end{enumerate}
%\item Let $\Theta := \lbrace \theta \in \Ccal(\Rbb^+): \theta(z) >0 \text{ and } \int_0^z \frac{dy}{\theta(y)} = \infty \text{ } \forall z\in\Rbb^+\rbrace$. We define the set of continuous functions that admit $\theta \in \Theta$ as a modulus of continuity by $$\MCcal_{\theta}(\Xcal) := \left\lbrace f\in \Ccal(\Xcal;\Rbb) : |f(x)-f(y)|^2 \leq \theta(d_\Xcal(x,y)^2) ~ \forall x,y \in \Xcal \right\rbrace.$$
\item We say a function $f:[0,T]\times \Rbb^d \times \Rbb^d \to \Rbb^d$ is in the space $\Lcal([0,T] \times \Rbb^d \times \Rbb^d)$, if there exists a constant $C>0$ such that for every $t\in[0,T]$ and $y \in \Rbb^d$ the function $(z \mapsto f(t,y,z)) \in \Ccal_{b,C}^{1,1}(\Rbb^d)$.
\item Let $(\Omega, \Fcal, \Fbb, \Pbb)$ be a generic complete filtered probability space with filtration $\Fbb = (\Fcal_t)_{t\in[0,T]}$ and $B = (B_t)_{t\in[0,T]}$ be $d$-dimensional Brownian motion defined on this probability space. Furthermore, we write $\Ebb[\cdot] := \Ebb_{\Pbb}[\cdot]$, if not mentioned differently.
\item $L^p(\Omega)$ denotes the Banach space of functions on the measurable space $(\Omega,\Fcal)$ integrable to some power $p$, $p\geq 1$.
\item $L^p(\Omega,\Fcal_t)$ denotes the space of $\Fcal_t$-measurable functions in $L^p(\Omega)$.
\item We define the weighted $L^p$-space over $\Rbb$ with weight function $\omega: \Rbb \to \Rbb$ as
\begin{align*}
	L^p(\Rbb;\omega) := \left\lbrace f:\Rbb \to\Rbb \text{ measurable : } \int_{\Rbb} |f(y)|^p \omega(y) dy <\infty \right\rbrace.
\end{align*} 
%\item We denote by $W^{1,p}(U)$ the space of Sobolev differentiable functions in $L^p(U)$, where $U\subseteq \Rbb$ is an open subset and $p\geq 1$.
\item Let $f: \Rbb^d \to \Rbb^d$ be a (weakly) differentiable function. Then we denote by $\partial_y f(y):= \frac{\partial f}{\partial y} (y)$ its first (weak) derivative evaluated at $y \in \Rbb^d$ and $\partial_k$ is the Jacobian in the direction of the $k$-th variable.
\item We denote the Doléan-Dade exponential for a progressive process $Y$ with respect to the corresponding Brownian integral if well-defined for $t\in[0,T]$ by $$\Ecal\left( \int_0^t Y_u dB_u \right) := \exp \left\lbrace \int_0^t Y_u dB_u - \frac{1}{2} \int_0^t \Vert Y_u \Vert^2du \right\rbrace.$$
\item We define $B_t^x := x + B_t$, $t\in [0,T]$, for any Brownian motion $B$. 
%\item For any normed space $\mathcal{X}$ we denote its corresponding norm by $\Vert \cdot \Vert_{\mathcal{X}}$; the Euclidean norm is denoted by $|\cdot |$.
%\item Let $\Pcal(\Rbb)$ be the set of probability measures on $(\Rbb, \Bcal(\Rbb))$. We denote the set of probability measures with finite first moment by $$\Pcal_1(\Rbb) := \left\lbrace \mu \in \Pcal(\Rbb) : \int_{\Rbb} |x| \mu(dx) < \infty \right\rbrace$$
%\item Let $\delta_x$ be the Dirac-measure in $x\in\Rbb$ and $\Kcal$ the Kantorovich metric, $$\Kcal(\lambda,\nu) := \sup_{h\in\Lip_1(\Rbb)} \left\vert \int_\Rbb h(x)(\lambda-\nu)(dx)\right\vert, \quad \lambda,\nu \in \Pcal_1(\Rbb).$$
\item We write $E_1(\theta) \lesssim E_2(\theta)$ for two mathematical expressions $E_1(\theta),E_2(\theta)$ depending on some parameter $\theta$, if there exists a constant $C>0$ not depending on $\theta$ such that $E_1(\theta) \leq C E_2(\theta)$.
%\item We denote by $L^{X}$ the local time of the stochastic process $X$ and furthermore by $\int_s^t \int_{\Rbb} f(u,y) L^X(du, dy)$ for suitable $f$ the local-time space integral as introduced in \cite{Eisenbaum_LocalTimeIntegral} and extended in \cite{MeyerBrandisBanosDuedahlProske_ComputingDeltas}.
\item We denote the Wiener transform of some $Z \in L^2(\Omega,\Fcal_T)$ in $f \in L^2([0,T])$ by
\begin{align*}
	\Wcal(Z)(f) := \Ebb \left[ Z \Ecal\left(\int_0^T f(s) dB_s \right) \right].
\end{align*}
\end{itemize}

\section{Results derived from the general mean-field SDE}\label{sec:Recall}
	In this section we recall sufficient conditions on $b$ and $\varphi$ such that $\bbar$ as defined in \eqref{eq:MFDrift} fulfills the corresponding assumptions for existence, uniqueness, and regularity properties of strong solutions required in \cite{Bauer_MultiDim} and \cite{Bauer_StrongSolutionsOfMFSDEs}.  These results will subsequently be applied in \Cref{sec:Solution,sec:Regularity} in order to weaken the assumptions on $\varphi$ such that mean-field SDE \eqref{lebesgueMFSDE} has a Malliavin differentiable strong solution and to show strong differentiability of this unique strong solution under sufficient conditions on $b$ and $\varphi$. We start by giving the definitions of some frequently used assumptions. \par 
	Let $f:[0,T] \times \Rbb^d \times \Rbb^d \to \Rbb^d$ be a measurable function. The function $f$ is said to be of linear growth, if there exists a constant $C>0$ such that for all $t\in[0,T]$ and $y,z\in \Rbb^d$,
	\begin{align}\label{linearGrowthMF} 
		\Vert f(t,y,z) \Vert &\leq C(1+\Vert y \Vert+\Vert z\Vert).
	\end{align}
	We say $f$ is continuous in the third variable (uniformly with respect to the first and second variable), if for all $z_1 \in \Rbb^d$ and $\varepsilon>0$ there exists a $\delta>0$ such that for all $t\in [0,T]$ and  $y\in \Rbb^d$
	\begin{align}\label{continuityThirdMF}
		\left( \forall z_2 \in \Rbb^d: \Vert z_1-z_2 \Vert < \delta \right) \Rightarrow \Vert f(t,y,z_1) - f(t,y,z_2) \Vert < \varepsilon.
	\end{align}
	Moreover, we say $f$ admits a modulus of continuity in the third variable, if there exists $\theta\in \lbrace \vartheta \in \Ccal(\Rbb_+;\Rbb): \vartheta(z) >0 \text{ and } \int_0^z \frac{dy}{\vartheta(y)} = \infty \text{ } \forall z\in\Rbb_+\rbrace$ such that for all $t\in[0,T]$ and $y,z_1,z_2\in\Rbb^d$
	\begin{align}\label{modulusOfContinuityMean}
		\Vert f(t,y,z_1) - f(t,y,z_2)\Vert^2 \leq \theta \left(\Vert z_1-z_2\Vert^2\right).
	\end{align}
	The function $f$ is said to be Lipschitz continuous in the second, respectively third, variable (uniformly with respect to the other two variables), if there exists a constant $C>0$ such that for all $t\in[0,T]$ and $y_1,y_2,z \in \Rbb^d$
	\begin{align}\label{lipschitzSecondMF}
		\left\Vert f(t,y_1,z) - f(t,y_2,z) \right\Vert \leq C \Vert y_1-y_2\Vert,
	\end{align}
	respectively, such that for all $t\in[0,T]$ and $y,z_1,z_2 \in \Rbb^d$
	\begin{align}\label{lipschitzThirdMF}
		\left\Vert f(t,y,z_1) - f(t,y,z_2) \right\Vert \leq C \Vert z_1-z_2\Vert.
	\end{align}
	Concluding, we say the function $f$ is Lipschitz continuous in the second and third variable (uniformly with respect to the first variable), if it fulfills the Lipschitz assumptions \eqref{lipschitzSecondMF} and \eqref{lipschitzThirdMF}, i.e. there exists a constant $C>0$ such that for all $t\in[0,T]$ and $y_1,y_2,z_1,z_2\in\Rbb^d$
	\begin{align}\label{eq:lipschitzSecondThird}
		\left\Vert f(t,y_1,z_1) - f(t,y_2,z_2) \right\Vert \leq C \left( \Vert y_1 - y_2 \Vert + \Vert z_1 - z_2 \Vert \right).
	\end{align}
	Note that when we talk about (Lipschitz) continuity in a certain variable, we always understand the continuity to hold uniformly with respect to the other variables. \par \vspace{0.3cm}
	
	We start by deriving sufficient conditions on $b$ and $\varphi$ from \cite{Bauer_MultiDim} and \cite{Bauer_StrongSolutionsOfMFSDEs} for existence and uniqueness of solutions of mean-field SDE \eqref{lebesgueMFSDE}. For detailed definitions of the notions weak and strong solution as well as pathwise uniqueness and uniqueness in law - as used subsequently - we refer the reader to these same papers.% \cite{Bauer_StrongSolutionsOfMFSDEs}.
	\par

	From \cite[Theorems 3.7]{Bauer_MultiDim} we obtain in the following corollary the assumptions on $b$ and $\varphi$ to ensure the existence of a strong solution of \eqref{lebesgueMFSDE}.	
		
	\begin{corollary}%\label{cor:weakSolutionRecall}
		Let $b:[0,T] \times \Rbb^d \times \Rbb^d \to \Rbb^d$ be a measurable function of at most linear growth \eqref{linearGrowthMF} and continuous in the third variable \eqref{continuityThirdMF}. Furthermore, assume that $\varphi:[0,T] \times \Rbb^d \times \Rbb^d \to \Rbb^d$ is a measurable functional which is of at most linear growth \eqref{linearGrowthMF} and Lipschitz continuous in the third variable \eqref{lipschitzThirdMF}. Then mean-field SDE \eqref{lebesgueMFSDE} has a strong solution. \par
		If in addition $b$ admits a modulus of continuity in the third variable \eqref{modulusOfContinuityMean}, the solution is pathwisely unique.
	\end{corollary}
	
 Concerning Malliavin differentiability of the solution we obtain from \cite[Theorem 4.1]{Bauer_MultiDim}:

	\begin{corollary}\label{cor:strongSolutionRecall}
		Let $b$ be a bounded measurable function which is continuous in the third variable \eqref{continuityThirdMF}. Furthermore, assume that $\varphi$ is a measurable functional which is of at most linear growth \eqref{linearGrowthMF} and Lipschitz continuous in the third variable \eqref{lipschitzThirdMF}. Then mean-field SDE \eqref{lebesgueMFSDE} has a Malliavin differentiable strong solution.
	\end{corollary}
	
	\begin{remark}
			In the one-dimensional case, $d=1$, \Cref{cor:strongSolutionRecall} can be generalized in the following way due to \cite{Bauer_StrongSolutionsOfMFSDEs}. Let $(b\diamond \varphi)$ allow for a decomposition of the form
	\begin{align}\label{formDriftMF}
		(b\diamond \varphi)(t,y,\mu) & := \bhat\left(t,y,\int_\Rbb \hat{\varphi}(t,y,z) \mu(dz) \right) + \btilde\left(t,y,\int_\Rbb \tilde{\varphi}(t,y,z) \mu(dz) \right),
	\end{align}
	where the drift $\hat{b}$ is merely measurable and bounded and the functional $\hat{\varphi}$ is of linear growth \eqref{linearGrowthMF} and Lipschitz continuous in the third variable \eqref{lipschitzThirdMF}. Moreover, the drift $\tilde{b}$ is of at most linear growth \eqref{linearGrowthMF} and Lipschitz continuous in the second variable \eqref{lipschitzSecondMF} whereas the functional $\tilde{\varphi}$ is of at most linear growth \eqref{linearGrowthMF} and Lipschitz continuous in the second and third variable \eqref{eq:lipschitzSecondThird}. Then, mean-field SDE \eqref{lebesgueMFSDE} has a Malliavin differentiable unique strong solution and the Malliavin derivative admits for $0\leq s \leq t \leq T$ the representation
	\begin{align}\label{eq:derivativeMalliavin}
		D_sX_t^x = \exp \left\lbrace - \int_s^t \int_{\Rbb} b \left(u,y, \int_\Rbb \varphi(u,y,z) \Pbb_{X_u^x}(dz) \right) L^{X^x}(du,dy) \right\rbrace.
	\end{align}
	Here, $L^{X^x}(du,dy)$ denotes integration with respect to local time of $X^x$ in time and space, see \cite{MeyerBrandisBanosDuedahlProske_ComputingDeltas} and \cite{Eisenbaum_LocalTimeIntegral} for more details. If in addition $b$ is continuously differentiable with respect to the second and third variable and $\varphi$ is continuously differentiable with respect to the second variable, representation \eqref{eq:derivativeMalliavin} can be written as
	\begin{align*}%\label{eq:derivativeMalliavinRegular}
		D_sX_t^x = &\exp \left\lbrace \int_s^t \partial_2 b\left(u, X_u^x, \int_\Rbb \varphi(u,X_u^x,z) \Pbb_{X_u^x}(dz) \right) \right.\\
		&\qquad +\left. \partial_3 b\left(u, X_u^x, \int_\Rbb \varphi(u,X_u^x,z) \Pbb_{X_u^x}(dz) \right) \int_\Rbb \partial_2 \varphi(u,X_u^x, z) \Pbb_{X_u^x}(dz)  du \right\rbrace. \notag
	\end{align*}
	Here, $\partial_2$ and $\partial_3$ denotes the derivative with respect to the second and third variable, respectively.
	\end{remark}

	Next we state a result on the regularity of a strong solution of \eqref{lebesgueMFSDE} in its initial condition which is due to \cite[Theorem 4.3]{Bauer_MultiDim}.

	\begin{corollary}\label{cor:SobolevDifferentiableRecall}
		Let $b$ be a bounded measurable function which is Lipschitz continuous in the third variable \eqref{lipschitzThirdMF}. Furthermore, assume that $\varphi$ is a measurable functional which is of at most linear growth \eqref{linearGrowthMF} and Lipschitz continuous in the third variable \eqref{lipschitzThirdMF}. Then, the unique strong solution $(X_t^x)_{t\in[0,T]}$ of mean-field SDE \eqref{lebesgueMFSDE} is Sobolev differentiable in the initial condition $x$.
	\end{corollary}
	
	\begin{remark}
		In the one-dimensional case, $d=1$, we further get due to \cite[Theorem 3.3 \& Proposition 3.4]{Bauer_StrongSolutionsOfMFSDEs} for $(b\diamond \varphi)$ allowing for a decomposition \eqref{formDriftMF} that the first variation process $(\partial_x X_t^x)_{t\in[0,T]}$ has for almost all $x \in K$, where $K\subset \Rbb$ is a compact subset, the representation
		\begin{small}
		\begin{align}\label{eq:RegRepresentationFV}
		\partial_x X_t^x &= \exp \left\lbrace - \int_0^t \int_\Rbb (b\diamond \varphi)\left(s,y, \Pbb_{X_s^x}\right) L^{X^x}(ds,dy) \right\rbrace \\
		&\quad+ \int_0^t \exp \left\lbrace - \int_u^t \int_\Rbb (b\diamond \varphi) \left(s,y,\Pbb_{X_s^x} \right) L^{X^x}(ds,dy) \right\rbrace \partial_x (b\diamond \varphi) \left(s,y, \Pbb_{X_u^x} \right)\Big\vert_{y= X_u^x} du. \notag
	\end{align}
		\end{small}
	Moreover, for  $0\leq s \leq t \leq T$ the following relationship with the Malliavin Derivative holds:
	\begin{align}\label{eq:RegRelationDerivatives}
		\partial_x X_t^x = D_sX_t^x \partial_x X_s^x + \int_s^t D_uX_t^x \partial_x (b\diamond \varphi) \left(u,y, \Pbb_{X_u^x} \right) \Big\vert_{y=X_u^x} du \,.
	\end{align}
	\end{remark}

	Furthermore, the unique strong solution is Hölder continuous in time and the initial condition which is due to \cite[Theorem 4.12]{Bauer_MultiDim}.
	
	\begin{corollary}\label{cor:Hölder}
		Let $b$ be a bounded measurable function which is Lipschitz continuous in the third variable \eqref{lipschitzThirdMF}. Furthermore, assume that $\varphi$ is a measurable functional which is of at most linear growth \eqref{linearGrowthMF} and Lipschitz continuous in the third variable \eqref{lipschitzThirdMF}. Let $(X_t^x)_{t\in[0,T]}$ be the unique strong solution of mean-field SDE \eqref{lebesgueMFSDE}. Then for every compact subset $K\subset \Rbb^d$ there exists a constant $C > 0$ such that for all $s,t\in[0,T]$ and $x,y \in K$,
	\begin{align}\label{eq:RegHoelderContinuity}
		\Ebb [\Vert X_t^x - X_s^y \Vert^2] \leq C(|t-s| + \Vert x-y\Vert^2).
	\end{align}
	In particular, there exists a continuous version of the random field $(t,x) \mapsto X_t^x$ with Hölder continuous trajectories of order $\alpha < \frac{1}{2}$ in $t\in[0,T]$ and $\alpha<1$ in $x\in \Rbb^d$.
\end{corollary}

	Finally, from \cite[Theorem 5.1]{Bauer_MultiDim} we get the following Bismut-Elworthy-Li type formula under the same assumptions as in \Cref{cor:SobolevDifferentiableRecall}.

	\begin{corollary}\label{cor:BismutFormulaRecall}
		Let $b$ be a bounded measurable function which is Lipschitz continuous in the third variable \eqref{lipschitzThirdMF}. Furthermore, assume that $\varphi$ is a measurable functional which is of at most linear growth \eqref{linearGrowthMF} and Lipschitz continuous in the third variable \eqref{lipschitzThirdMF}. Furthermore, let $\Phi \in L^{2p}(\Rbb^d;\omega_T)$ with $p:= \frac{1+\varepsilon}{\varepsilon}$, $\varepsilon>0$ sufficiently small with regard to \Cref{lem:RegBoundsSolution}, and $\omega_T(y) := \exp \left\lbrace - \frac{\Vert y \Vert^2}{4T} \right\rbrace$. Then, the expectation functional $ \Ebb\left[ \Phi(X_t^x) \right]$ is Sobolev differentiable in the initial condition and the derivative $\partial_x \Ebb\left[ \Phi(X_T^x) \right]$ admits for almost all $x\in K$, where $K \subset \Rbb^d$ is a compact subset, the representation
	\begin{footnotesize}
	\begin{align}\label{eq:RegDelta}
		\partial_x \Ebb[\Phi(X_T^x)] = \Ebb \left[ \Phi(X_T^x) \left( \int_0^T a(s) \partial_x X_s^x + \partial_x (b\diamond \varphi) \left(s,y,\Pbb_{X_s^x} \right)\vert_{y=X_s^x} \int_0^s a(u) du dB_s \right) \right],
	\end{align}
	\end{footnotesize}
	where $a: \Rbb  \to \Rbb$ is any bounded, measurable function such that
	\begin{align*}
		\int_0^T a(s) ds = 1.
	\end{align*}
	\end{corollary}

\section{Existence and Uniqueness of Solutions }\label{sec:Solution}

The results in \Cref{sec:Recall} presume Lipschitz continuity of the function $\varphi$. In this section we are interested in showing existence and uniqueness of strong solutions under weakened regularity assumptions on $\varphi$. In particular, this will allow to consider mean-field SDEs where the drift depends on the solution law in form of indicator and distribution functions, respectively.

\begin{theorem}\label{extensionSolution}
	Let $b:[0,T] \times \Rbb^d \times \Rbb^d \to \Rbb^d$ be of at most linear growth \eqref{linearGrowthMF} and continuous in the third variable \eqref{continuityThirdMF}. Further, let $\varphi:[0,T] \times \Rbb^d \times \Rbb^d \to \Rbb^d$ be of at most linear growth \eqref{linearGrowthMF}. Then mean-field SDE \eqref{lebesgueMFSDE} has a strong solution. \par 
	If in addition $b$ is Lipschitz continuous in the third variable \eqref{lipschitzThirdMF}, the solution is unique.
\end{theorem}
\begin{proof}
	This proof is organized as follows. First we introduce a sequence $\lbrace Y^n \rbrace_{n\geq 1}$ of solutions to mean-field SDE \eqref{eq:RegMainMcKeanVlasov} with approximating coefficients and show that we can find a probability space $(\Omega, \Fcal, \Pbb)$ such that the equivalent sequence $\lbrace X^k \rbrace_{k\geq 1}$ of $\lbrace Y^n \rbrace_{n\geq 1}$ on this probability space converges in $L^2(\Omega)$ to some stochastic process $X$. We then prove that $X^k$ further converges weakly in $L^2(\Omega)$ to a solution of \eqref{lebesgueMFSDE} and thus by uniqueness of the limit $X$ is a weak solution of mean-field SDE \eqref{lebesgueMFSDE}. Afterwards we conclude the existence of a strong solution and prove uniqueness of the solution. \par 
	By standard arguments using mollifiers, we can define sequences $\lbrace b_n \rbrace_{n\geq 1}$ and $\lbrace \varphi_n\rbrace_{n\geq 1}$ in $\Ccal_0^\infty([0,T] \times \Rbb^d \times \Rbb^d)$ such that $b_n$ converges to $b$ and $\varphi_n$ converges to $\varphi$, respectively, pointwise in $(t,y,z) \in [0,T]\times \Rbb^d \times \Rbb^d$ a.e. with respect to the Lebesgue measure. We denote the original functions $b$ and $\varphi$ by $b_0$ and $\varphi_0$, respectively. Due to continuity assumption \eqref{continuityThirdMF} on the coefficient $b$, we can further assume that the family of coefficients $\lbrace b_n \rbrace_{n\geq 0}$ is pointwisely equicontinuous in the third variable, i.e. for every $\varepsilon>0$ and $z_1 \in \Rbb^d$ exists a $\delta >0$ such that for all $n\geq 0$, $t\in [0,T]$, and $y \in \Rbb^d$ we get
	\begin{align}\label{continuityThirdEquiApprox}
		\left( \forall z_2 \in \Rbb^d: \Vert z_1-z_2 \Vert< \delta \right) \Rightarrow \Vert b_n(t,y,z_1) - b_n(t,y,z_2) \Vert < \varepsilon.
	\end{align}
Then, by \Cref{cor:strongSolutionRecall}, for $n\geq 1$ mean-field SDEs
	\begin{align}\label{eq:RegApproximatingMFSDE}
		\begin{split}
		dY_t^n &= b_n\left(t,Y_t^n,\int_{\Rbb^d} \varphi_n\left(t,Y_t^n,z\right) \Pbb_{Y_t^n}(dz) \right) dt + dW_t, ~ t\in[0,T],\\
		Y_0^n &= x \in \Rbb^d,
		\end{split}
	\end{align}
	where $W= (W_t)_{t\in[0,T]}$ is Brownian motion, have unique strong solutions $\lbrace Y^n \rbrace_{n\geq 1}$ on some complete probability space $(\tilde{\Omega}, \tilde{\Fcal}, \tilde{\Pbb})$. Moreover, due to \Cref{lem:RegBoundsSolution} there exists some constant $C>0$ such that
	\begin{enumerate}[(i)]
		\item $\sup_{n\geq 1} \sup_{0\leq t \leq T} \EPpO{Y_t^n}{\tilde{\Pbb}}{2} \leq C(1+\Vert x\Vert^2)$,
		\item $\sup_{n\geq 1} \sup_{0\leq s \leq t \leq T; t-s \leq h} \EPpO{Y_t^n-Y_s^n}{\tilde{\Pbb}}{2} \leq Ch$.
	\end{enumerate}
	Next, we show that the properties (i) and (ii) imply the assumptions of \Cref{thm:SkorohodRepresentationTheorem} and thus there exists a subsequence $\lbrace n_k \rbrace_{k\geq 1} \subset \Nbb$ and a sequence of stochastic processes $\lbrace (X_t^k)_{t\in[0,T]} \rbrace_{k\geq 1}$ defined on a probability space $(\Omega, \Fcal, \Pbb)$ such that the finite dimensional distributions of the processes $Y^{n_k}$ and $X^k$ coincide for every $k\geq 1$, c.f. \Cref{rem:finiteDistribution}, and $X_t^k$ converges in probability to $X_t$ as $k$ goes to infinity. Note first that the stochastic processes $\lbrace Y^n \rbrace_{n\geq 1}$ are almost surely continuous as a solution of mean-field SDE \eqref{eq:RegApproximatingMFSDE}. Furthermore, we get by Chebyshev's inequality that due to (i)
	\begin{align*}
		\tilde{\Pbb}( \Vert Y_t^n \Vert > K) &\leq \frac{1}{K^2} \EPpO{Y_t^n}{\tilde{\Pbb}}{2} \leq \frac{1}{K^2} C(1+\Vert x \Vert^2 ), \quad K>0,
	\end{align*}
	and thus
	\begin{align*}
		\lim_{K \to \infty} \lim_{n \to \infty} \sup_{t\in [0,T]} \tilde{\Pbb}( \Vert Y_t^n \Vert > K) \leq \lim_{K \to \infty} \lim_{n \to \infty} \sup_{t\in [0,T]} \frac{1}{K^2} C(1+\Vert x \Vert^2 ) = 0.
	\end{align*} 
	Equivalently, we get due to property (ii) that for every $\varepsilon > 0$
	\begin{align*}
		\tilde{\Pbb}(\Vert Y_t^n - Y_s^n \Vert > \varepsilon) \leq \frac{1}{\varepsilon^2} \EPpO{Y_t^n - Y_s^n}{\tilde{\Pbb}}{2} \leq \frac{C h}{\varepsilon^2},
	\end{align*}
	and thus 
	\begin{align*}
		\lim_{h\to 0} \lim_{n\to \infty} \sup_{\vert t-s \vert \leq h} \tilde{\Pbb}(\Vert Y_t^n - Y_s^n \Vert > \varepsilon) \leq \lim_{h\to 0} \lim_{n\to \infty} \sup_{\vert t-s \vert \leq h} \frac{C h}{\varepsilon^2} = 0.
	\end{align*}
	 Consequently, the assumptions of \Cref{thm:SkorohodRepresentationTheorem} are fulfilled. For the sake of readability, we assume in the following without loss of generality that $n_k = k$. Further note that due to the uniform integrability of $\lbrace \Vert X_t^k \Vert^2 \rbrace$ by property (i), we get that for every $t\in [0,T]$ the sequence $\lbrace X_t^k \rbrace_{k\geq 1}$ converges to $X_t$ in $L^2(\Omega)$. Due to property (ii) we further get in connection with Kolmogorov's continuity theorem that $(X_t)_{t\in[0,T]}$ can be assumed to have almost surely continuous path. Using approximation by Riemann sums, we further have that $$\int_0^t b_k\left(s,X_s^k, \int_{\Rbb^d} \varphi_k \left(s,X_s^k,z\right) \Pbb_{X_s^k}(dz) \right) ds$$ and $$ \int_0^t b_k \left(s,Y_s^k, \int_{\Rbb^d} \varphi_k \left(s,Y_s^k,z\right) \Pbb_{Y_s^k}^k(dz) \right) ds$$ have the same distribution for every $k\geq 1$. Again by virtue of \Cref{thm:SkorohodRepresentationTheorem} we get that 
	 \begin{align*}
	 	B_t^k := X_t^k - \int_0^t b_k \left(s,X_s^k, \int_{\Rbb^d} \varphi_k \left(s,X_s^k,z\right) \Pbb_{X_s^k}(dz) \right) ds
	 \end{align*}
	 is $d$-dimensional Brownian motion on the probability space $(\Omega, \Fcal, \Pbb)$ and thus $X^k$ solves \eqref{eq:RegApproximatingMFSDE} on the stochastic basis $(\Omega, \Fcal, \Fbb, \Pbb, B^k)$.\par 
	Let us define the stochastic differential equation
	\begin{align}\label{helpMFSDE}
		d\overline{X}_t = b\left(t,\overline{X}_t, \int_{\Rbb^d} \varphi\left(t,\overline{X}_t,z\right) \Pbb_{X_t}(dz) \right) dt + dB_t,~ t\in[0,T], ~\overline{X}_0 = x \in \Rbb^d.
	\end{align}
	Due to the result of Veretennikov given in \cite{veretennikov1981strong}, SDE \eqref{helpMFSDE} has a unique strong solution on the probability space $(\Omega, \Fcal, \Pbb)$. Therefore it is left to show that for every $t \in [0,T]$ the sequence $\lbrace X_t^k \rbrace_{k\geq 1}$ converges weakly in $L^2(\Omega)$ to $\overline{X}_t$. Indeed, if this holds true, we get by the uniqueness of the limit that $\Pbb_{X_t} = \Pbb_{\overline{X}_t}$ for all $t\in[0,T]$ and consequently mean-field SDE \eqref{lebesgueMFSDE} and \eqref{helpMFSDE} coincide. Hence, we have found a weak solution of \eqref{lebesgueMFSDE}. In order to prove weak convergence in $L^2(\Omega)$ we use the Wiener transform and show for every $f\in L^2([0,T])$,
	\begin{align*}
		\left\Vert\Wcal\left(X_t^k \right)(f) - \Wcal\left(\overline{X}_t \right)(f)\right\Vert \xrightarrow[n\to\infty]{} 0.
	\end{align*}
	Using inequality
	\begin{align}\label{eq:RegExponentialInequality}
		|e^x - e^y| \leq |x-y|(e^x+e^y), \quad \forall x,y\in\Rbb,
	\end{align}	
	Burkholder-Davis-Gundy's inequality and Minkowski's integral inequality, we get for $p:= \frac{1+ \varepsilon}{\varepsilon}$, $\varepsilon>0$ sufficiently small with respect to \Cref{lem:RegBoundsSolution}, that
	\begin{align*}
		&\left\Vert \Wcal\left(X_t^k\right)(f) - \Wcal\left(\overline{X}_t\right)(f)\right\Vert \\
		&\quad \leq \Ebb \left[ \left\Vert B_t^x\right\Vert \left| \Ecal\left( \int_0^T b_k\left(t,B_t^x,\int_{\Rbb^d} \varphi_k\left(t,B_t^x,z\right) \Pbb_{X_t^k}(dz) \right) + f(t) dB_t \right) \right.\right.\\
		&\qquad \left.\left. - \Ecal\left( \int_0^T b\left(t,B_t^x,\int_{\Rbb^d} \varphi\left(t, B_t^x,z\right) \Pbb_{X_t}(dz)\right) +f(t) dB_t \right) \right| \right]\\
		&\quad \lesssim \left( \int_0^T \Ebb \left[ \left\Vert b_k\left(t,B_t^x,\int_{\Rbb^d} \varphi_k\left(t,B_t^x,z\right) \Pbb_{X_t^k}(dz) \right) \right.\right.\right.\\
		&\qquad - \left.\left. \left. b\left(t,B_t^x,\int_{\Rbb^d} \varphi\left(t, B_t^x,z\right) \Pbb_{X_t}(dz) \right)\right\Vert^p \right]^{\frac{2}{p}} dt \right)^{\frac{1}{2}} + C_n \\
		&\quad =: A_n + C_n,
	\end{align*}
	where
	\begin{align*}
		C_n := \int_0^T &\Ebb\left[ \left| \left\Vert b_k\left(t,B_t^x,\int_{\Rbb^d} \varphi_k\left(t,B_t^x,z\right) \Pbb_{X_t^k}(dz) \right)+f(t)\right\Vert^2 \right. \right.\\
		&\quad \left. \left.- \left\Vert b\left(t,B_t^x, \int_{\Rbb^d} \varphi\left(t,B_t^x,z\right) \Pbb_{X_t}(dz) \right)+f(t)\right\Vert^2 \right|^p \right]^{\frac{1}{p}} dt.
	\end{align*}
	We show using dominated convergence that $A_n$ converges to $0$ as $k$ tends to infinity. Since the family of coefficients $\lbrace b_k \rbrace_{k\geq 0}$ is pointwisely equicontinuous in the third variable \eqref{continuityThirdEquiApprox}, it suffices to show that for all $t\in [0,T]$ and $y \in \Rbb^d$
	\begin{align*}
		\left\Vert \int_{\Rbb^d} \varphi_k\left(t,y,z\right) \Pbb_{X_t^k}(dz) - \int_{\Rbb^d} \varphi\left(t,y,z\right) \Pbb_{X_t}(dz) \right\Vert  &\xrightarrow[k\to\infty]{}0, \text{ and} \\
		\Ep{\overline{b}_k\left(t,B_t^x,\Pbb_{X_t} \right) - \overline{b}\left(t,B_t^x,\Pbb_{X_t} \right)}{p} &\xrightarrow[k\to\infty]{}0.
	\end{align*}
	The second convergence is an immediate consequence of the definition of $b_k$, \Cref{lem:RegBoundsSolution}, and dominated convergence. Thus, it remains to show the first convergence. Let $\delta>0$. Since $\varphi_k$ is of at most linear growth \eqref{linearGrowthMF} for all $k\geq 0$, we get by (i) that
	\begin{align*}
		&\sup_{k\geq 0} \Ebb \left[ \left\Vert \varphi_k\left(t,y, X_t^k \right) \right\Vert \right] \leq C \left(1+ \Vert y \Vert + \sup_{k\geq 0} \Eabs{X_t^k} \right) \leq C_1,
	\end{align*}
	where $C_1>0$ is some constant independent of $k\geq 0$. Hence, due to dominated convergence we can find $N_1\in \Nbb$ sufficiently large such that 
	\begin{align*}
		\sup_{k\geq N_1} \Ebb\left[\left\Vert \varphi_k \left(t,y, X_t \right) - \varphi\left(t,y, X_t\right) \right\Vert \right]<\frac{\delta}{3}.
	\end{align*}
	Note further that for $\varepsilon>0$ sufficiently small with respect to \Cref{lem:RegBoundsSolution},
	\begin{align*}
		\sup_{k\geq 0} \Ebb\left[\Ecal\left(\int_0^T b_k\left(t,B_t^x, \int_{\Rbb^d} \varphi_k(t, B_t^x, z) \Pbb_{X_t^k}(dz)\right) dB_t \right)^{1+\varepsilon}\right]^{\frac{1}{1+\varepsilon}} \leq C_2 < \infty,
	\end{align*}
	where $C_2>0$ is some constant. Thus we can find by Girsanov's theorem and again by dominated convergence an integer $N_2\in \Nbb$ such that 
	\begin{align*}
		&\sup_{m,k\geq N_2} \Ebb\left[ \left\Vert \varphi_k(t,y,X_t^k) - \varphi_m(t,y,X_t^k) \right\Vert \right] \\
		&\quad \leq \sup_{m,k\geq N_2} C_2 \Ebb\left[\left\Vert \varphi_k(t,y,B_t^x)-\varphi_m(t,y,B_t^x) \right\Vert^p \right]^{\frac{1}{p}}<\frac{\delta}{3},
	\end{align*}
	where $p:= \frac{1+\varepsilon}{\varepsilon}$.
	Therefore, using Minkowski's and Hölder's inequality we get for $N:= \max\lbrace N_1, N_2 \rbrace$ and $k\geq N$
	\begin{align*}
		&\left\Vert \Ebb\left[ \varphi_k(t,y,X_t^k) \right] - \Ebb\left[ \varphi(t,y,X_t) \right] \right\Vert \\
		&\quad \leq \Ebb\left[ \left\Vert \varphi_k(t,B_t^x, X_t^k) - \varphi_N(t,B_t^x,X_t^k) \right\Vert \right] + D_k \\
		&\qquad + \Ebb\left[ \left\Vert\varphi_N(t,y,X_t) - \varphi(t,y, X_t) \right\Vert \right] \\
		&\quad \leq D_k + \frac{2\delta}{3},
	\end{align*}
	where
	\begin{align*}
		D_k:= \left\Vert \Ebb\left[\varphi_N(t,y,X_t^k) \right] - \Ebb\left[ \varphi_N(t,y,X_t) \right] \right\Vert.
	\end{align*}
	Since $\varphi_N$ is smooth and has compact support by the definition of mollification, $\varphi_N$ is also bounded. Hence, using the fact that $X_t^k$ converges in distribution to $X_t$ for every $t\in[0,T]$, we can find $k\geq N$ sufficiently large such that $D_k < \frac{\delta}{3}$. Analogously one can show that $C_k$ converges to $0$ as $k$ tends to infinity and therefore, $X$ is a weak solution of the mean-field SDE \eqref{lebesgueMFSDE}.  Due to the proof of \cite[Theorem 3.7]{Bauer_MultiDim} we get as a direct consequence the existence of a strong solution of mean-field equation \eqref{lebesgueMFSDE} for the more general class of functionals $\varphi$.\par 

	Let $(\Omega, \Fcal, \Fbb, \Pbb, X, B)$ and $(\hat{\Omega}, \hat{\Fcal}, \hat{\Fbb}, \hat{\Pbb}, Y, W)$ be two weak solutions of mean-field SDE \eqref{lebesgueMFSDE} and assume that the drift coefficient $b$ is Lipschitz continuous in the third variable \eqref{lipschitzThirdMF}. In the following we show that $X$ and $Y$ have the same law, i.e. $\Pbb_X = \hat{\Pbb}_Y$. For the sake of readability we just consider the case $x=0$. The general case follows analogously. From \cite{Bauer_MultiDim} we know that there exist measures $\Qbb$ and $\hat{\Qbb}$ such that under these measures the processes $X$ and $Y$ are Brownian motions, respectively. Similar to the proofs of \cite[Theorem 3.7]{Bauer_MultiDim} and \cite[Theorem 2.7]{Bauer_StrongSolutionsOfMFSDEs} we use the idea of Li and Min in the proof of Theorem 4.2 in \cite{LiMin_WeakSolutions} and define the equivalent probability measure $\widetilde{\Qbb}$ by
 \begin{align*}
 	\frac{\dd \widetilde\Qbb}{\dd \Pbb} := \Ecal \left( - \int_0^T \left( \bbar\left(t,X_t,\Pbb_{X_t} \right) - \bbar\left(t,X_t, \hat{\Pbb}_{Y_t} \right) \right) dB_t^X \right).
 \end{align*}
 Due to \cite{Bauer_MultiDim} and \cite{Bauer_StrongSolutionsOfMFSDEs}, $$\hat{\Pbb}_{(Y,W)} = \widetilde{\Qbb}_{(X,B)}.$$ Thus, it is left to show that $\sup_{t\in[0,T]} \Kcal \left(\widetilde\Qbb_{X_t}, \Pbb_{X_t} \right) = 0$, from which we conclude that $\sup_{t\in[0,T]} \Kcal \left(\hat{\Pbb}_{Y_t}, \Pbb_{X_t} \right) = 0$ and hence $\frac{d \widetilde\Qbb}{d \Pbb}  = 1$. Consequently, $\hat{\Pbb}_{(Y,W)} = \Pbb_{(X,B)}.$ Here, $\Kcal$ denotes the Kantorovich metric defined by
 \begin{align*}
 	\Kcal( \mu, \nu) = \sup_{h \in \Lip_1(\Rbb^d;\Rbb)} \left\vert \int_{\Rbb^d} h(x) (\mu - \nu)(dx) \right\vert, ~\mu, \nu \in \Pcal_1(\Rbb^d).
 \end{align*} 
 Using Hölder's inequality for $p:= \frac{1+\varepsilon}{\varepsilon}$, where $\varepsilon>0$ is sufficiently small with regard to \Cref{lem:RegBoundsSolution}, inequality \eqref{eq:RegExponentialInequality}, Burkholder-Davis-Gundy's inequality, and the Lipschitz continuity of $b$ we get
 \begin{align*}
 	&\Kcal \left(\widetilde\Qbb_{X_t}, \Pbb_{X_t}\right) = \sup_{h \in \Lip_1(\Rbb^d;\Rbb)} \left\vert \Ebb_{\widetilde\Qbb} \left[h(X_t)-h(0) \right] - \Ebb\left[h(X_t)-h(0) \right] \right\vert \\
 	&\quad \leq \Ebb \left[ \left\Vert X_t \right\Vert \left\vert \Ecal \left( - \int_0^t \left( \bbar\left(s,X_s,\Pbb_{X_s}\right) - \bbar\left(s,X_s, \hat{\Pbb}_{Y_s}\right) \right) dB_s \right) - 1 \right\vert\right] \\
 	&\quad \lesssim \Ebb \left[ \left\vert \Ecal \left( - \int_0^t \left( \bbar\left(s,X_s,\Pbb_{X_s}\right) - \bbar\left(s,X_s, \hat{\Pbb}_{Y_s}\right) \right) dB_s \right) - 1 \right\vert^{\frac{2(1+\varepsilon)}{2+\varepsilon}} \right]^{\frac{2+\varepsilon}{2(2+\varepsilon)}} \\
 	&\quad \lesssim \Ebb \left[ \left\vert \int_0^t \left( \bbar\left(s,X_s,\Pbb_{X_s} \right) - \bbar\left(s,X_s, \hat{\Pbb}_{Y_s} \right) \right) dB_s \right.\right. \\
 	&\qquad \left. \left. + \frac{1}{2} \int_0^t \left\Vert \bbar\left(s,X_s,\Pbb_{X_s}\right) - \bbar\left(s,X_s, \hat{\Pbb}_{Y_s}\right) \right\Vert^2 ds \right\vert^{2p} \right]^{\frac{1}{2p}} \\
 	&\quad \lesssim \Ebb \left[ \left\vert \int_0^t \left\Vert \bbar\left(s,X_s,\Pbb_{X_s}\right) - \bbar\left(s,X_s, \hat{\Pbb}_{Y_s} \right) \right\Vert^2 ds \right\vert^{p} \right]^{\frac{1}{2p}} \\
 	&\qquad + \Ebb \left[ \left\vert \int_0^t \left\Vert \bbar\left(s,X_s,\Pbb_{X_s} \right) - \bbar\left(s,X_s, \hat{\Pbb}_{Y_s} \right) \right\Vert^2 ds \right\vert^{2p} \right]^{\frac{1}{2p}} \\
 	&\quad \lesssim \max_{q=1,2} \Ebb \left[ \left( \int_0^t \left\Vert \int_{\Rbb^d} \varphi(s,X_s,z) \left(\Pbb_{X_s} - \hat{\Pbb}_{Y_s}\right)(dz) \right\Vert^2 ds \right)^{qp} \right]^{\frac{1}{2p}} \\
 	&\quad = \max_{q=1,2} \Ebb \left[ \left( \int_0^t \left\Vert \int_{\Rbb^d} \varphi(s,B_s,z) \left(\Pbb_{X_s} - \hat{\Pbb}_{Y_s}\right)(dz) \right\Vert^2 ds \right)^{qp} \right. \\
 	&\qquad \left. \times \Ecal\left(-\int_0^s \bbar(u,B_u,\Pbb_{X_u}) dB_u \right) \right]^{\frac{1}{2p}} \\
 	&\quad \lesssim \max_{q=1,2} \Ebb \left[ \left( \int_0^t \left\Vert \int_{\Rbb^d} \varphi(s,B_s,z) \left(\Pbb_{X_s} - \hat{\Pbb}_{Y_s}\right)(dz) \right\Vert^2 ds \right)^{qp^2} \right]^{\frac{1}{2p^2}}.
 \end{align*}
 Equivalent to the steps before we get using $\sup_{t\in[0,T]}\EW{\Vert \bbar(t,B_t,\mu_t)\Vert^2} <\infty$ for all $\mu \in \Ccal([0,T];\Pcal_1(\Rbb^d))$ that
 \begin{align*}
 	 &\left\Vert \int_{\Rbb^d} \varphi(s,B_s,z) \left(\Pbb_{X_s} - \hat{\Pbb}_{Y_s}\right)(dz) \right\Vert^2\\
 	 &\quad = \left\Vert \Ebb \left[ \varphi(s,y,X_s) \right] - \Ebb_{\hat\Pbb}\left[\varphi(s,y,Y_s) \right] \right\Vert_{y= B_s }^2 \\
 	 &\quad = \Ebb\left[ \left\Vert \varphi(s,y,B_s) \right\Vert \right. \\
 	 &\qquad \times \left.\left\vert \Ecal\left( - \int_0^s \bbar(u,B_u,\Pbb_{X_u}) dB_u \right) - \Ecal\left( - \int_0^s \bbar(u,B_u,\hat\Pbb_{Y_u}) dB_u \right) \right\vert \right]_{y= B_s}^2 \\
 	 &\quad \lesssim \left( 1+ B_s \right)^2 \Ebb\left[\left\vert \int_0^s \left( \bbar(u,B_u,\Pbb_{X_u}) - \bbar(u,B_u,\hat\Pbb_{Y_u}) \right) dB_u \right.\right.\\
 	 &\qquad \left.\left. +\frac{1}{2} \int_0^s \left( \left\Vert\bbar(u,B_u,\Pbb_{X_u})\right\Vert^2 - \left\Vert\bbar(u,B_u,\hat\Pbb_{Y_u})\right\Vert^2\right) du \right\vert^{2p} \right]^{\frac{1}{p}} \\
 	 &\quad \lesssim \left( 1+ B_s \right)^2 \Ebb\left[ \left( \int_0^s \left\Vert \bbar(u,B_u,\Pbb_{X_u}) - \bbar(u,B_u,\hat\Pbb_{Y_u}) \right\Vert^2 du\right)^p \right]^{\frac{1}{p}}  \\
 	 &\quad +  \left( 1+ B_s \right)^2 \Ebb\left[\left( \int_0^s \left\vert \left\Vert\bbar(u,B_u,\Pbb_{X_u})\right\Vert^2 - \left\Vert\bbar(u,B_u,\hat\Pbb_{Y_u})\right\Vert^2 \right\vert du \right)^{2p} \right]^{\frac{1}{p}} \\
 	  &\quad \lesssim \left( 1+ B_s \right)^2 \Ebb\left[ \left( \int_0^s \left\Vert \bbar(u,B_u,\Pbb_{X_u}) - \bbar(u,B_u,\hat\Pbb_{Y_u}) \right\Vert^2 du\right)^p \right]^{\frac{1}{p}}  \\
 	  &\lesssim \left( 1+ B_s \right)^2 \Ebb \left[ \left( \int_0^s \left\Vert \int_{\Rbb^d} \varphi(s,B_s,z) \left(\Pbb_{X_s} - \hat{\Pbb}_{Y_s}\right)(dz) \right\Vert^{2} du \right)^p \right]^{\frac{1}{p}}.
 \end{align*}
 Applying the $L^{p^2}(\Omega)$ norm on both sides yields
 \begin{align*}
 	&\Ebb \left[ \left\Vert \int_\Rbb \varphi(s,B_s,z) \left(\Pbb_{X_s} - \hat{\Pbb}_{Y_s}\right)(dz) \right\Vert^{2p^2} \right]^{\frac{1}{p^2}}\\
 	&\lesssim \Ebb \left[ \left( 1+ B_s \right)^{2p^2} \right] \Ebb \left[ \left( \int_0^s \left\Vert \int_{\Rbb^d} \varphi(s,B_s,z) \left(\Pbb_{X_s} - \hat{\Pbb}_{Y_s}\right)(dz) \right\Vert^{2} du \right)^{p^2} \right]^{\frac{1}{p^2}} \\
 	& \lesssim \int_0^s \Ebb \left[ \left\Vert \int_{\Rbb^d} \varphi(s,B_s,z) \left(\Pbb_{X_s} - \hat{\Pbb}_{Y_s}\right)(dz) \right\Vert^{2p^2} \right]^{\frac{1}{p^2}} du
 \end{align*}
 Using a Grönwall argument yields that
 \begin{align*}
 	\Ebb \left[ \left\Vert \int_{\Rbb^d} \varphi(s,B_s,z) \left(\Pbb_{X_s} - \hat{\Pbb}_{Y_s}\right)(dz) \right\Vert^{2p^2} \right]^{\frac{1}{p^2}} = 0.
 \end{align*}
 In particular, 
 \begin{align*}
 	\left\Vert \int_{\Rbb^d} \varphi(s,B_s,z) \left(\Pbb_{X_s} - \hat{\Pbb}_{Y_s}\right)(dz) \right\Vert = 0, \quad \Pbb\text{-a.s.}
 \end{align*}
  and consequently, $\Kcal \left(\widetilde\Qbb_{X_t}, \Pbb_{X_t}\right) =0$.
\end{proof}

 Due to \cite[Theorem 4.1]{Bauer_MultiDim} we immediately get Malliavin differentiability of the strong solution of mean-field equation \eqref{lebesgueMFSDE} for a more general class of functionals $\varphi$.

\begin{theorem}\label{thm:Malliavin}
	Let $b:[0,T] \times \Rbb^d \times \Rbb^d \to \Rbb^d$ be bounded and continuous in the third variable \eqref{continuityThirdMF}. Further, let $\varphi:[0,T] \times \Rbb^d \times \Rbb^d \to \Rbb^d$ be of at most linear growth \eqref{linearGrowthMF}. Then, the strong solution of mean-field SDE \eqref{lebesgueMFSDE} is Malliavin differentiable.
\end{theorem}

\begin{remark}
	In the one-dimensional case, $d=1$, the class of drift coefficients $b$ and functionals $\varphi$ can be further extended in order to obtain Malliavin differentiability of the strong solution. Consider the decomposition
	\begin{align}\label{eq:RegFormDrift}
		(b\diamond \varphi)\left(t,y,\mu\right) := \bhat\left(t,y,\int_\Rbb \hat{\varphi}(t,y,z) \mu(dz) \right) + \btilde\left(t,y,\int_\Rbb \tilde{\varphi}(t,y,z) \mu(dz) \right),
	\end{align}
	where the drift $\bhat$ is merely measurable and bounded and the functional $\hat{\varphi}$ is merely measurable and of linear growth whereas $\btilde$ and $\tilde{\varphi}$ are of linear growth \eqref{linearGrowthMF} and Lipschitz continuous in the second variable \eqref{lipschitzSecondMF}. If $b$ is continuous in the third variable \eqref{continuityThirdMF}, the strong solution of mean-field SDE \eqref{lebesgueMFSDE} is Malliavin differentiable due to \cite[Theorem 2.12]{Bauer_StrongSolutionsOfMFSDEs}.
\end{remark}

\begin{example}
	Let $b:[0,T] \times \Rbb \times \Rbb \to \Rbb$ be a measurable and bounded function which is continuous in the third variable \eqref{continuityThirdMF}. Furthermore, define the functional $\varphi(t,y,z) = \mathbbm{1}_{\lbrace z\leq u \rbrace}$, where $u \in \Rbb$ is some parameter. Then, the mean-field stochastic differential equation
	\begin{align*}%\label{eq:CDFmfsde1}
		dX_t^x = b\left(t,X_t^x, F_{X_t^x}(u)\right) dt + dB_t, ~ t\in[0,T], ~ X_0^x = x \in \Rbb,
	\end{align*}
	where $F_{X_t^x}$ denotes the cumulative distribution function of $X_t^x$, has a Malliavin differentiable strong solution due to \Cref{thm:Malliavin}. If $b$ is Lipschitz continuous in the third variable \eqref{lipschitzThirdMF}, the solution is unique. Note that it is also possible to choose $u=t$ or $u = y$, where the later one yields the mean-field SDE
	\begin{align*}%\label{eq:CDFmfsde2}
		dX_t^x = b\left(t,X_t^x, F_{X_t^x}(X_t^x)\right) dt + dB_t,~ t\in[0,T], ~ X_0^x = x \in \Rbb.
	\end{align*}
\end{example}
\bigskip

%In the case of $\varphi(t,y,z) = \overline{\varphi}(t,z)$, i.e.
%\begin{align}\label{expectationMFSDE}
%	d X_t^x = b\left( t,X_t^x, \Ebb[\overline{\varphi}(t,X_t^x)] \right) dt + dB_t, \quad t\in [0,T], \quad X_0^x = x \in \Rbb,
%\end{align}
%where $\overline{\varphi}:\Rbb \to \Rbb$, 
Using It\^o's formula we are able to extend our results on mean-field SDE \eqref{lebesgueMFSDE} to more general diffusion coefficients. For notational simplicity we just consider the time-homogenous and one-dimensional case. However the time-inhomogeneous and multi-dimensional cases can be shown analogously.

\begin{theorem}%\label{generalDiffusion}
	Consider the time-homogeneous mean-field SDE
	\begin{align}\label{generalizedMFSDE}
		dX_t^x = b\left(X_t^x,\int_{\Rbb} \varphi(X_t^x,z) \Pbb_{X_t^x}(dz)\right) dt + \sigma(X_t^x) dB_t,~ t\in[0,T], ~ X_0^x = x \in \Rbb,
	\end{align}
	with measurable drift $b: \Rbb \times \Rbb \to \Rbb$, functional $\varphi:\Rbb \times \Rbb \to \Rbb$, and volatility $\sigma: \Rbb \to \Rbb$. Moreover, let  $\Lambda: \Rbb \to \Rbb$ be a twice continuously differentiable bijection with derivatives $\Lambda'$ and $\Lambda''$, such that for all $y\in\Rbb$,
	\begin{align*}
		\Lambda'(y) \sigma(y) = 1,
	\end{align*}
	as well as $\Lambda^{-1}$ is Lipschitz continuous. Suppose that $(b^* \diamond \varphi^*): \Rbb \times \Pcal_1(\Rbb) \to \Rbb$, defined by
	\begin{small}
	\begin{align*}
		&(b^*\diamond \varphi^*)(y,\mu) := \\
		&\quad \Lambda'\left(\Lambda^{-1}(y)\right)b\left(\Lambda^{-1}(y),\int_{\Rbb^d} \varphi(\Lambda^{-1}(y),\Lambda^{-1}(z)) \mu(dz)\right) + \frac{1}{2} \Lambda''\left(\Lambda^{-1}(y)\right) \sigma\left(\Lambda^{-1}(y)\right)^2,
	\end{align*}
	\end{small}
	fulfills the assumptions of \Cref{extensionSolution} and \Cref{thm:Malliavin}, respectively. Then, there exists a (Malliavin differentiable) strong solution $(X_t^x)_{t\in[0,T]}$ of \eqref{generalizedMFSDE}. If moreover $b^*$ is Lipschitz continuous in the second variable \eqref{lipschitzThirdMF}, the solution is unique.
\end{theorem}
\begin{proof}
	Since $(b^* \diamond \varphi^*)$ satisfies the conditions of \Cref{extensionSolution} and \Cref{thm:Malliavin}, respectively, mean-field SDE
	\begin{align*}
		dZ_t^x = b^*\left(Z_t^x,\int_\Rbb \varphi^*\left(Z_t^x, z\right) \Pbb_{Z_t^x}(dz)\right)dt + dB_t,~ t\in[0,T],~ Z_0^x = \Lambda(x),
	\end{align*}
	has a (Malliavin differentiable) (unique) strong solution. Thus $X_t^x := \Lambda^{-1}(Z_t^x)$ is a (unique) strong solution of \eqref{generalizedMFSDE} by the application of It\^{o}'s formula, and since $\Lambda^{-1}$ is Lipschitz continuous, $X^x$ is Malliavin differentiable.
\end{proof}

We conclude this section by applying our existence result on solutions of mean-field SDEs to construct solutions of ODEs. More precisely, consider the mean-field SDE
\begin{align}\label{MfOde}
	dX_t^x = b(t,\Ebb[X_t^x])dt + dB_t,~ t\in[0,T],~ X_0^x = x\in \Rbb^d,
\end{align}
i.e.~the drift coefficient only depends on the solution via the expectation $\Ebb[X_t^x]$. By \Cref{extensionSolution}, mean-field SDE \eqref{MfOde} has a strong solution if $b:[0,T] \times \Rbb^d \to \Rbb^d$ is of at most linear growth and continuous in the second variable. Now, by taking expectation on both sides, we loose the randomness and get that $u(t) := \Ebb[X_t^x]$ solves the ODE
	\begin{align}\label{eq:ODE}
		d\,u(t) = b(t,u(t)) dt,~ t\in[0,T],~ u(0) = x\in \Rbb^d.
	\end{align}
We thus have developed a probabilistic approach to the following version of the theorem on existence of solutions of ODEs by Carathéodory, see e.g.~\cite[Theorem 1.1]{Persson_AGeneralizationOfCaratheodorysExistenceTheoremForODEs} or for a direct proof \cite[Chapter II, Theorem 3.2]{Reid_ODE}:

\begin{theorem}
	Let $b:[0,T] \times \Rbb^d \to \Rbb^d$ be of at most linear growth and continuous in the second variable, i.e. $b$ fulfills the corresponding assumptions \eqref{linearGrowthMF} and \eqref{continuityThirdMF}. Furthermore, let $\left(X_t^x\right)_{t\in[0,T]}$ be a strong solution of \eqref{MfOde}. Then $u(t):=\Ebb[X_{t}^x]$ is a solution of ODE \eqref{eq:ODE}.
\end{theorem}

%%%%%%%%%%%%%%%%%%%%%%%%%%%%%%%%%%%%%%%%%%%%%%%%%%%%%%%%%%%%%%%%%%%%%%%%%%%%%%%%%%%%%%
%%%%%%%%%%%%%%%%%%%%%%%%%%%%%%%%%%%%%%%%%%%%%%%%%%%%%%%%%%%%%%%%%%%%%%%%%%%%%%%%%%%%%%

\section{Regularity in the initial value}\label{sec:Regularity}
The aim of this section is to study the regularity of a strong solution of mean-field SDE \eqref{lebesgueMFSDE} as a function in its initial condition. More precisely, we investigate under which assumptions on $b$ and $\varphi$ the strong solution $X_t^x$ of \eqref{lebesgueMFSDE} is not just Sobolev differentiable but continuously differentiable as a function in $x$. These results will then be used to develop the Bismut-Elworthy-Li formula \eqref{eq:RegDelta}.

\subsection{Strong Differentiability}
	First recall that due to \Cref{cor:SobolevDifferentiableRecall} the unique strong solution $X^x$ of mean-field SDE \eqref{lebesgueMFSDE} is Sobolev differentiable under the assumption that $b$ is measurable, bounded, and Lipschitz continuous in the third variable \eqref{lipschitzThirdMF}, and $\varphi$ is measurable, of at most linear growth \eqref{linearGrowthMF}, and Lipschitz continuous in the third variable \eqref{lipschitzThirdMF}. Our aim is to find sufficient assumptions on $b$ and $\varphi$ such that the unique strong solution $X^x$ of \eqref{lebesgueMFSDE} is continuously differentiable in the initial condition.

\begin{proposition}\label{thm:strongDerivative}
	Suppose $b, \varphi \in\Cfrak([0,T]\times \Rbb^d \times \Rbb^d)$. Let $(X_t^{x})_{t\in[0,T]}$ be the unique strong solution of mean-field SDE \eqref{lebesgueMFSDE}. Then for every compact subset $K\subset \Rbb^d$ there exists some constant $C>0$ such that for every $t \in [0,T]$ and $x,y\in K$
	\begin{align*}
		\Eabs{\partial_x X_t^x - \partial_y X_t^y} \leq C \Vert x-y \Vert.
	\end{align*}
	In particular, the map $x \mapsto X_t^x$ is continuously differentiable for every $t\in[0,T]$ and for every $1\leq p < \infty$
	\begin{align}\label{eq:boundDerivative}
		\sup_{t\in [0,T]} \sup_{x\in K} \Ep{\partial_x X_t^x}{p} < \infty.
	\end{align}
\end{proposition}
\begin{proof}
	Since $X^x$ is Sobolev differentiable by \Cref{cor:SobolevDifferentiableRecall} and
	\begin{align*}
		\sup_{t\in [0,T]} \esssup_{x\in K} \Ep{\partial_x X_t^x}{p} < \infty,
	\end{align*}
	by \cite[Lemma 4.13]{Bauer_MultiDim}, it suffices to show that $\partial_x X^x$ is almost surely continuous in $x\in K$. Note that we can choose an element of the equivalence class of weak derivatives $\partial_x X^x$ such that \eqref{eq:boundDerivative} holds. For the remainder of this proof we just consider this particular element and denote it without loss of generality by $\partial_x X^x$. Let $x,y \in K$ and $t\in[0,T]$ be arbitrary. Note that the first variation process $\partial_x X^x$ has the representation
	\begin{align*}
		\partial_x X_t^x = 1+ \int_0^t \partial_2 b(s,X_s^x, \rho(X_s^x)) \partial_x X_s^x + \partial_3 b(s,X_s^x, \rho(X_s^x)) \partial_x \rho(X_s^x) ds,
	\end{align*}
where $\rho(X_t^x) := \int_{\Rbb^d} \varphi(t,X_t^x,z) \Pbb_{X_t^x}(dz).$ Thus, using Minkowski's and Hölder's inequalities we get
	%\begin{small}
	\begin{align*}
		&\Eabs{\partial_x X_t^x - \partial_y X_t^y} \\
		&\quad \leq \int_0^t \Eabs{\partial_2 b(s,X_s^x, \rho(X_s^x)) \partial_x X_s^x - \partial_2 b(s,X_s^y, \rho(X_s^y)) \partial_y X_s^y} \\
		&\qquad + \Eabs{\partial_3 b(s,X_s^x, \rho(X_s^x)) \partial_x \rho(X_s^x) - \partial_3 b(s,X_s^y, \rho(X_s^y)) \partial_y \rho(X_s^y)} ds \\
		&\quad \lesssim \int_0^t \Ep{\partial_2 b(s,X_s^x, \rho(X_s^x)) - \partial_2 b(s,X_s^y, \rho(X_s^y))}{2} \Ep{\partial_x X_s^x}{2}\\
		&\qquad + \Ep{\partial_3 b(s,X_s^x, \rho(X_s^x)) - \partial_3 b(s,X_s^y, \rho(X_s^y))}{2} \Ep{\partial_x \rho(X_s^x)}{2}\\
		&\qquad + \Eabs{\partial_x X_s^x - \partial_y X_s^y} + \Eabs{\partial_x \rho(X_s^x) - \partial_y \rho(X_s^y)} ds\\
		&\quad \lesssim \int_0^t \left( \Ep{X_s^x - X_s^y}{2} + \Ep{\rho(X_s^x) - \rho(X_s^y)}{2} \right) \\
		&\qquad \times \left( \Ep{\partial_x X_s^x}{2} + \Ep{\partial_x \rho(X_s^x)}{2} \right) \\
		&\qquad + \Eabs{\partial_x X_s^x - \partial_y X_s^y} + \Eabs{\partial_x \rho(X_s^x) - \partial_y \rho(X_s^y)} ds.
	\end{align*}
	%\end{small}
	Using the assumptions on $\varphi$ we get
	\begin{align}\label{eq:rhoContinuity}
		\Ep{\rho(X_t^x) - \rho(X_t^y)}{2} &= \Ep{ \EW{\varphi(t,z_1,X_t^x)  - \varphi(t,z_2,X_t^y)}_{z_1=X_t^x; z_2= X_t^y}}{2} \notag\\
		&\lesssim \Ep{ \Eabs{z_1-z_2 \right\Vert + \left\Vert X_t^x  - X_t^y}_{z_1=X_t^x; z_2= X_t^y}}{2} \notag\\
		&\leq \Ep{X_t^x - X_t^y}{2}.
	\end{align}
	Furthermore, using the chain rule we have that
	\begin{align*}
		&\Ep{\partial_x \rho(X_t^x)}{2} \\
		&\quad = \Ep{\EW{\partial_2 \varphi(t,z,X_t^x)}_{z=X_t^x} \partial_x X_t^x + \EW{\partial_3 \varphi(t,z,X_t^x) \partial_x X_t^x}_{z=X_t^x}}{2} \\
		&\quad \lesssim \Ep{\partial_x X_t^x}{2} + \Eabs{\partial_x X_t^x} \leq \Ep{\partial_x X_t^x}{2}.
	\end{align*}
	Equivalently we obtain that 
	\begin{align*}
		&\Eabs{\partial_x \rho(X_t^x) - \partial_y \rho(X_t^y)} \\
		&\quad \leq \Eabs{\EW{\partial_2 \varphi(t,z,X_t^x)}_{z=X_t^x} \partial_x X_t^x - \EW{\partial_2 \varphi(t,z,X_t^y)}_{z=X_t^y} \partial_y X_t^y}\\
		&\qquad + \Eabs{\EW{\partial_3 \varphi(t,z,X_t^x) \partial_x X_t^x}_{z=X_t^x} - \EW{\partial_3 \varphi(t,z,X_t^y) \partial_y X_t^y}_{z=X_t^y}} \\
		&\quad \leq \Eabs{\EW{\partial_2 \varphi(t,z_1,X_t^x) - \partial_2 \varphi(t,z_2,X_t^y)}_{z_1= X_t^x; z_2 =X_t^y} \right\Vert \left\Vert \partial_x X_t^x}\\
		&\qquad + \Eabs{\partial_x X_t^x - \partial_y X_t^y\right\Vert^p \left\Vert \EW{\partial_2 \varphi(t,z,X_t^y)}_{z=X_t^y} }\\
		&\qquad + \Eabs{\Eabs{\partial_3 \varphi(t,z_1,X_t^x) - \partial_3 \varphi(t,z_2,X_t^y)\right\Vert \left\Vert \partial_x X_t^x}_{z_1=X_t^x; z_2=X_t^y}} \\
		&\qquad + \Eabs{\Eabs{ \partial_x X_t^x - \partial_y X_t^y \right\Vert \left\Vert \partial_3 \varphi(t,z,X_t^y)}_{z=X_t^y}}\\
		&\lesssim \Ep{\EW{\partial_2 \varphi(t,z_1,X_t^x) - \partial_2 \varphi(t,z_2,X_t^y)}_{z_1= X_t^x; z_2 =X_t^y}}{2} \Ep{\partial_x X_t^x}{2}\\
		&\qquad + \Eabs{\partial_x X_t^x - \partial_y X_t^y} \\
		&\qquad + \Eabs{\Ebb\left[\left\Vert \partial_3 \varphi(t,z_1,X_t^x) - \partial_3 \varphi(t,z_2,X_t^y)\right\Vert^2 \right]^{\frac{1}{2}}_{z_1=X_t^x; z_2=X_t^y}} \Ep{\partial_x X_t^x}{2}\\
		&\qquad + \Eabs{ \partial_x X_t^x - \partial_y X_t^y}\\ 
		&\quad \lesssim \Ep{\partial_x X_t^x}{2}\Ep{X_t^x - X_t^y}{2} + \Eabs{\partial_x X_t^x - \partial_y X_t^y}.
	\end{align*}
	Thus, in combination with \eqref{eq:boundDerivative} we get
	\begin{align*}
		&\Eabs{\partial_x X_t^x - \partial_y X_t^y} \lesssim \int_0^t \Ep{X_s^x - X_s^y}{2} + \Eabs{\partial_x X_s^x - \partial_y X_s^y} ds.
	\end{align*}
	Using equation \eqref{eq:RegHoelderContinuity}, we get
	\begin{align*}
		&\Eabs{\partial_x X_t^x - \partial_y X_t^y} \lesssim \vert x - y \vert + \int_0^t \Eabs{\partial_x X_s^x - \partial_y X_s^y} ds.
	\end{align*}
	Finally, since $\Eabs{\partial_x X_s^x - \partial_y X_s^y}$ is integrable over $[0,T]$ and Borel measurable, we can apply Jones' generalization of Grönwall's inequality \cite[Lemma 5]{Jones_FundamentalInequalities} to get
	\begin{align*}
		\Eabs{\partial_x X_t^x - \partial_y X_t^y} \lesssim \vert x-y \vert.
	\end{align*}
	Thus, $\partial_x X^x$ has an almost surely continuous version in $x\in K$ by Kolmogorov's continuity theorem and consequently $X^x$ is continuously differentiable for every $t\in[0,T]$.
\end{proof}

\subsection{Bismut-Elworthy-Li formula}%\label{sec:RegSBEL}

In this subsection we turn our attention to the Bismut-Elworthy-Li formula \eqref{eq:RegDelta}. With the help of the approximating sequence defined in \eqref{eq:RegApproximatingMFSDEExp} we show in the one-dimensional case, i.e. $d=1$, that $\partial_x \Ebb[\Phi(X_T^x)]$ exists in the strong sense for functionals $\Phi$ merely satisfying some integrability condition, i.e. we show that $\Ebb[\Phi(X_T^x)]$ is continuously differentiable.
%Note that in Mathematical Finance the expectation functional $\Ebb[\Phi(X_T^x)]$ can be interpreted as the price/value of a European option with payoff function $\Phi$. The sensitivity $\frac{\partial}{\partial x} \Ebb[\Phi(X_T^x)]$ is then known as Delta.

%Here we extend the results of \cite{Bauer_StrongSolutionsOfMFSDEs} in the special case of mean-field SDE \eqref{expectationMFSDE}, i.e.
%\begin{align*}
%	dX_t^x = b(t,X_t^x,\Ebb[\varphi(X_t^x)])dt + dB_t, ~ X_0^x = x \in \Rbb, ~ t\in [0,T].
%\end{align*}

\begin{lemma}\label{lem:RegRepDelta2}
	Consider $d=1$. Let $(b \diamond \varphi)$ admit a decomposition \eqref{eq:RegFormDrift} and let $b, \varphi \in \Lcal([0,T] \times \Rbb \times \Rbb)$. Further, let $(X_t^x)_{t\in [0,T]}$ be the unique strong solution of mean-field SDE \eqref{lebesgueMFSDE} and $\Phi \in \Ccal_b^{1,1}(\Rbb)$. Then $\Ebb[\Phi(X_t^x)] \in \Ccal^1(\Rbb)$ and 
	\begin{align}\label{representationDerivative}
		\partial_x \Ebb \left[ \Phi(X_t^x) \right] = \Ebb \left[\Phi'(X_t^x) \partial_x X_t^x \right],
	\end{align}
	where $\Phi'$ denotes the first derivative of $\Phi$ and $\partial_x X_t^x$ is the first variation process of $X_t^x$ as given in \eqref{eq:RegRepresentationFV}.
\end{lemma}

In order to proof \Cref{lem:RegRepDelta2}, we need to define a sequence of mean-field equations similar to \cite{Bauer_StrongSolutionsOfMFSDEs} whose unique strong solutions approximate the unique strong solution of \eqref{lebesgueMFSDE}, where $(b \diamond \varphi)$ fulfills the assumptions of \Cref{lem:RegRepDelta2}. More precisely, by standard approximation arguments there exist sequences
\begin{align}\label{eq:RegApproximatingDrift}
	b_n:= \btilde_n + \bhat_n, \quad \text{ and } \quad \varphi_n := \tilde{\varphi}_n + \hat{\varphi}_n, \quad n\geq 1,
\end{align}
where $b_n, \varphi_n \in \Ccal_0^{\infty}( [0,T] \times \Rbb \times \Rbb)$ with \[\sup_{n\geq 1} \left( \Vert \btilde_n\Vert_{\infty} + \Vert \tilde{\varphi}_n\Vert_{\infty} \right) \leq C < \infty\] and \[\sup_{n\geq 1}\left( \vert \bhat_n(t,y,z) \vert + \vert \hat{\varphi}_n(t,y,z) \vert \right) \leq C(1+\vert y \vert + \vert z \vert)\] for every $t\in[0,T]$ and $y,z \in \Rbb$, such that $b_n \to b$ and $\varphi_n \to \varphi$ in a.e. $(t,y,z) \in [0,T] \times \Rbb \times \Rbb$ with respect to the Lebesgue measure, respectively. The original drift coefficients $b$ and $\varphi$ are denoted by $b_0$ and $\varphi_0$, respectively. Furthermore, we can assume that there exists a constant $C>0$ independent of $n\in \Nbb$ such that \[b_n, \varphi_n \in \Lcal([0,T]\times \Rbb \times \Rbb),\] and that $\bhat_n$ and $\hat{\varphi}_n$ are Lipschitz continuous in the second variable \eqref{lipschitzSecondMF} for all $n\geq 0$. Under these conditions the corresponding mean-field SDEs, defined by
\begin{align}\label{eq:RegApproximatingMFSDEExp}
	dX_t^{n,x} &= b_n\left(t,X_t^{n,x}, \int_{\Rbb} \varphi_n(t,X_t^{n,x}, z) \Pbb_{X_t^{n,x}}(dz) \right) dt + dB_t, ~t\in[0,T],\\
	X_0^{n,x} &= x \in \Rbb, \notag
\end{align}
have unique strong solutions which are Malliavin differentiable by \Cref{thm:Malliavin}. Likewise the strong solutions $\lbrace X^{n,x}\rbrace_{n\geq 0}$ are continuously differentiable with respect to the initial condition by \Cref{thm:strongDerivative}. Due to \Cref{cor:L2Convergence} we have that $(X_t^{n,x})_{t\in[0,T]}$ converges to $(X_t^x)_{t\in[0,T]}$ in $L^2(\Omega)$ as $n\to \infty$ and similar to \cite[Lemma 3.10]{Bauer_StrongSolutionsOfMFSDEs} one can show for any compact subset $K\subset \Rbb$ and $p\geq 1$ that
\begin{align}\label{eq:boundDerivativeUniform}
	\sup_{n\geq 0} \sup_{t\in[0,T]} \sup_{x\in K} \Ep{\partial_x X_t^{n,x}}{p} < \infty.
\end{align}

\begin{proof}[Proof of \Cref{lem:RegRepDelta2}]
Note first that $\Ebb \left[ \Phi(X_t^x) \right]$ is weakly differentiable by \Cref{cor:BismutFormulaRecall} and equation \eqref{representationDerivative} holds by \cite[Lemma 4.1]{Bauer_StrongSolutionsOfMFSDEs}. Hence it suffices to show that $\partial_x \Ebb[\Phi(X_t^x)]$ is continuous. In order to prove this we show that 
	\begin{align*}
		\Ebb[\Phi(X_t^{n,x})] &\xrightarrow[n\to\infty]{} \Ebb[\Phi(X_t^x)] \quad \forall x\in \Rbb, \text{ and}\\
		\Ebb \left[ \Phi'(X_t^{n,x}) \partial_x X_t^{n,x} \right] &\xrightarrow[n\to\infty]{} \Ebb \left[ \Phi'(X_t^x) \partial_x X_t^x \right] \quad \text{uniformly for } x\in  K,
	\end{align*}
	where $\lbrace(X_t^{n,x})_{t\in[0,T]} \rbrace_{n\geq 1}$ is the approximating sequence defined in \eqref{eq:RegApproximatingMFSDEExp} and $K \subset \Rbb$ is a compact subset. Note that $$\partial_x \Ebb[\Phi(X_t^{n,x})]= \Ebb \left[ \Phi'(X_t^{n,x}) \partial_x X_t^{n,x} \right]$$ is continuous in $x$ due to \Cref{thm:strongDerivative}. The first convergence follows directly by \Cref{rem:convergenceRho}. For the uniform convergence let $K\subset \Rbb$ be a compact set and define for $n\geq 0$
	\begin{align*}
		D_n(s,t,x) &:= \exp \left\lbrace -\int_s^t \int_{\Rbb} b_n(u,y,\varrho_u^{n,x}(y)) L^{B^x}(du,dy) \right\rbrace, \text{ and}\\
		E_n(x) &:= \Ecal\left( \int_0^T b_n(s,B_s^x,\varrho_s^{n,x}(B_s^x)) dB_s \right),
	\end{align*}
	where $\varrho_u^{n,x}(y) := \int_\Rbb \varphi(u,y,z) \Pbb_{X_u^{n,x}}(dz)$. In a first approximation we get using $\Vert \Phi' \Vert_{\infty} < \infty$ and representation \eqref{eq:RegRelationDerivatives} that
	\begin{align*}%\label{definitionAntx}
	\begin{split}
		&\left| \Ebb \left[ \Phi'(X_t^{n,x}) \partial_x X_t^{n,x} - \Phi'(X_t^x) \partial_x X_t^x \right] \right| \\
		&\quad \lesssim \Ebb \left[ \left| E_n(x) \left( D_n(0,t,x) + \int_0^t D_n(s,t,x) \partial_x b_n(s,y,\varrho_s^{n,x}(y))\vert_{y=B_s^x} ds \right) \right. \right.\\
		&\qquad -\left.\left. E_0(x) \left( D_0(0,t,x) + \int_0^t D_0(s,t,x) \partial_x b_n(s,y,\varrho_s^{n,x}(y))\vert_{y=B_s^x} ds \right) \right| \right] \\
		&\quad =: A_n(t,x).
	\end{split}
	\end{align*}
	Equivalently, we get using $\Vert \partial_3 \varphi \Vert_\infty < \infty$ that for every $t\in [0,T]$ and $y \in \Rbb$
	\begin{align*}
		\vert \partial_x \varrho_t^{n,x}(y)  - \partial_x \varrho_t^x(y) \vert = \left\vert \Ebb \left[ \partial_3 \varphi(t,y,X_t^{n,x}) \partial_x X_t^{n,x} - \partial_3 \varphi(t,y,X_t^x) \partial_x X_t^x \right] \right\vert  \lesssim A_n(t,x).
	\end{align*}
	Note furthermore that by \eqref{eq:boundDerivativeUniform} we have for every $y\in \Rbb$ that
	\begin{align}\label{eq:boundPartialXb}
		\vert \partial_x b_n(s,y,\varrho_s^{n,x}(y)) \vert &= \vert \partial_3 b_n(s,y,\varrho_s^{n,x}(y)) \partial_x \varrho_s^{n,x}(y) \vert \lesssim \left\vert \Ebb \left[ \partial_3 \varphi(t,y,X_t^{n,x}) \partial_x X_t^{n,x}\right] \right\vert \notag \\
		&\leq \Ebb[\vert \partial_x X_t^{n,x} \vert] < \infty,
	\end{align}
	and for every $p\geq 1$
	\begin{align}\label{eq:lipPartialXb}
		&\EpE{\partial_x b_n(t,y,\varrho_t^{n,x}(y))\vert_{y=B_t^x} - \partial_x b(t,y,\varrho_t^{x}(y))\vert_{y=B_t^x}}{p} \notag \\
		&\quad \lesssim \EpE{\partial_3 b_n(t,B_t^x,\varrho_t^{n,x}(B_t^x))- \partial_3 b(t,B_t^x,\varrho_t^{x}(B_t^x))}{p} \notag\\
		&\qquad + \EpE{\partial_x \varrho_t^{n,x}(B_t^x)  - \partial_x \varrho_t^x(B_t^x)}{p} \\
		&\quad \lesssim \EpE{\partial_3 b_n(t,B_t^x,\varrho_t^{n,x}(B_t^x))- \partial_3 b(t,B_t^x,\varrho_t^{x}(B_t^x))}{p} + A_n(t,x). \notag
	\end{align}
	Using Hölder's inequality, \eqref{eq:boundPartialXb}, \Cref{lem:RegBoundsSolution}, and \Cref{cor:RegBoundLocalTime} we can decompose $A_n(t,x)$ into
	\begin{small}
	\begin{align*}
		&A_n(t,x) \\
		&\lesssim \Ebb \left[ \left| E_n(x) - E_0(x) \right| \left| D_n(0,t,x) + \int_0^t D_n(s,t,x) \partial_x b_n(s,y,\varrho_s^{n,x}(y))\vert_{y=B_s^x} ds \right| \right] \\
		&\quad + \Ebb \left[ E_0(x) \left| D_n(0,t,x) - D_0(0,t,x) \right| \right] \\
		&\quad + \Ebb \left[ E_0(x) \left| \int_0^t  D_n(s,t,x) \partial_x b_n(s,y,\varrho_s^{n,x}(y))\vert_{y=B_s^x} \right.\right.\\
		&\qquad \left.\left. - D_0(s,t,x) \partial_x b(s,y,\varrho_s^{x}(y))\vert_{y=B_s^x} ds \right| \right] \\
		 &\lesssim \EpE{E_n(x) - E_0(x)}{q} + \EpE{D_n(0,t,x) - D_0(0,t,x)}{p} \\
		&\quad + \EpE{\int_0^t  D_n(s,t,x) \partial_x b_n(s,y,\varrho_s^{n,x}(y))\vert_{y=B_s^x} - D_0(s,t,x) \partial_x b(s,y,\varrho_s^{x}(y))\vert_{y=B_s^x} ds}{p}\\
		&=: F_n(x) + G_n(0,t,x) + H_n(t,x),
	\end{align*}
	\end{small}
	where $q:= \frac{2(1+\varepsilon)}{2+\varepsilon}$ and $p:= \frac{1+\varepsilon}{\varepsilon}$. Furthermore, we can bound $H_n(t,x)$ due to \Cref{cor:RegBoundLocalTime}, \eqref{eq:boundPartialXb}, and \eqref{eq:lipPartialXb} by
	\begin{small}
	\begin{align*}
		H_n(t,x) &\leq \int_0^t  \EpE{D_n(s,t,x) - D_0(s,t,x) \right|^p \left| \partial_x b_n(s,y,\varrho_s^{n,x}(y))\vert_{y=B_s^x}}{p} ds \\
		&\qquad + \int_0^t  \EpE{\partial_x b_n(s,y,\varrho_s^{n,x}(y))\vert_{y=B_s^x} - \partial_x b(s,y,\varrho_s^{x}(y))\vert_{y=B_s^x} \right|^p \left| D_0(s,t,x)}{p} ds \\
		&\lesssim \int_0^t G_n(s,t,x) ds + \int_0^t \EpE{\partial_3 b_n(s,B_s^x,\varrho_s^{n,x}(B_s^x))- \partial_3 b(s,B_s^x,\varrho_s^{x}(B_s^x))}{2p} ds \\
		&\qquad + \int_0^t A_n(s,x) ds \\
		&=: \int_0^t G_n(s,t,x) ds + \int_0^t K_n(s,x) ds + \int_0^t A_n(s,x) ds,
	\end{align*}
	\end{small}
	and thus
	\begin{align*}
		A_n(t,x) \leq C \left( F_n(x) + \sup_{s\in[0,t]} G_n(s,t,x) + \sup_{s\in[0,T]} K_n(s,x) \right) + C \int_0^t A_n(s,x) ds,
	\end{align*}
	for some constant $C >0$ independent of $t\in[0,T]$, $n\geq 0$, and $x\in K$. Consequently we get by Grönwall's inequality
	\begin{align*}
		A_n(t,x) \lesssim F_n(x) + G_n(0,t,x) + \sup_{s\in[0,T]} K_n(s,x) + \int_0^t \int_0^t G_n(s,r,x) ds dr.
	\end{align*}
	$F_n$ converges to $0$ uniformly in $x\in K$ by \Cref{cor:RegUniformConvergenceEcal}. Furthermore, we have that $G_n(s,t,x)$ is integrable over $t$ and $s$ by \Cref{cor:RegBoundLocalTime} and converges to $0$ uniformly in $x\in K$ by \Cref{cor:RegUniformConvergenceLocalTime}. Finally, we get due to $b \in \Lcal([0,T] \times \Rbb \times \Rbb)$ that
	\begin{align*}
		K_n(s,x) &\leq \int_0^t \EpE{\partial_3 b_n(s,B_s^x,\varrho_s^{n,x}(B_s^x)) - \partial_3 b_n(s,B_s^x,\varrho_s^x(B_s^x)}{2p} ds \\
		&\qquad + \int_0^t \EpE{\partial_3 b_n(s,B_s^x,\varrho_s^x(B_s^x)) - \partial_3 b(s,B_s^x,\varrho_s^x(B_s^x))}{2p} ds \\
		&\lesssim \int_0^t \left\vert \varrho_s^{n,x}(B_s^x) - \varrho_s^x(B_s^x) \right\vert ds \\
		&\qquad + \int_0^t \EpE{\partial_3 b_n(s,B_s^x,\varrho_s^x(B_s^x)) - \partial_3 b(s,B_s^x,\varrho_s^x(B_s^x))}{2p} ds
	\end{align*}
	Note first that due to \Cref{rem:convergenceRho} we have that $\left\vert \varrho_s^{n,x}(B_s^x) - \varrho_s^x(B_s^x) \right\vert$ converges uniformly in $s\in[0,T]$ and $x \in K$ to $0$ as $n$ goes to infinity. Moreover,
	\begin{align*}
		&\EpE{\partial_3 b_n(s,B_s^x,\varrho_s^x(B_s^x)) - \partial_3 b(s,B_s^x,\varrho_s^x(B_s^x))}{2p} \\
		&\quad = \left( \int_{\Rbb} \left| \partial_3 b_n\left(t,y,\varrho_s^x(y) \right) - \partial_3 b\left(t,y,\varrho_s^x(y) \right) \right|^{2p} \frac{1}{\sqrt{2\pi t}} e^{-\frac{(y-x)^2}{2t}} dy \right)^{\frac{2}{2p}} \\
		&\quad \leq e^{\frac{x^2}{2pt}} \left( \int_{\Rbb} \left| \partial_3 b_n\left(t,y,\varrho_s^x(y) \right) - \partial_3 b\left(t,y,\varrho_s^x(y) \right) \right|^{2p} \frac{1}{\sqrt{2\pi t}} e^{-\frac{y^2}{4t}} dy\right)^{\frac{2}{2p}},
	\end{align*}
	where we have used $e^{-\frac{(y-x)^2}{2t}} = e^{-\frac{y^2}{4t}} e^{-\frac{(y-2x)^2}{4t}} e^{\frac{x^2}{2t}} \leq e^{-\frac{y^2}{4t}} e^{\frac{x^2}{2t}}$. Furthermore, equivalent to \eqref{eq:rhoContinuity} we can find a constant $C>0$ by \Cref{cor:Hölder} such that for all $t\in[0,T]$ and $x,y \in K$
	\begin{align*}
		\left\vert \varrho_s^{x}(z) - \varrho_s^{y}(z) \right\vert \leq C|x-y|.
	\end{align*}
	Consequently the function $x \mapsto \varrho_s^{x}(y)$ is continuous uniformly in $t\in [0,T]$. Thus $\Lambda^K := \lbrace \varrho_s^{x}(y) : x\in K \rbrace \subset \Rbb$ is compact as an image of a compact set under a continuous function. Therefore due to the definition of the approximating sequence
	\begin{align*}
		\sup_{x\in K} \left| \partial_3 b_n(s,y,\varrho_s^{x}(y)) - \partial_3 b(s,y,\varrho_s^{x}(y)) \right| = \sup_{z \in \Lambda^K} \left| \partial_3 b_n(s,y,z) - \partial_3 b(s,y,z) \right| \xrightarrow[n\to\infty]{} 0,
	\end{align*}
	and hence $K_n(s,x)$ converges to $0$ uniformly in $s\in[0,T]$ and $x\in K$.
\end{proof}

We define the weight function $\omega_T: \Rbb \to \Rbb$ by
\begin{align}\label{eq:RegWeightFunction}
	\omega_T(y) := \exp \left\lbrace - \frac{|y|^2}{4T} \right\rbrace, \quad y \in \Rbb.
\end{align}

\begin{theorem}\label{thm:RegMainTheoremDelta}
	Consider $d=1$. Let $(b \diamond \varphi)$ admit a decomposition \eqref{eq:RegFormDrift} and let $b, \varphi \in \Lcal([0,T] \times \Rbb \times \Rbb)$. Further, let $(X_t^x)_{t\in [0,T]}$ be the unique strong solution of mean-field SDE \eqref{lebesgueMFSDE} and $\Phi \in L^{2p}(\Rbb; \omega_T)$, where $p:= \frac{1+\varepsilon}{\varepsilon}$, $\varepsilon>0$ sufficiently small with respect to \Cref{lem:RegBoundsSolution} and $\omega_T: \Rbb \to \Rbb$ as defined in \eqref{eq:RegWeightFunction}. Then
	\begin{align*}
		u(x) := \Ebb \left[ \Phi(X_T^x) \right]
	\end{align*}
	is continuously differentiable in $x \in \Rbb$ and the derivative takes the form
	\begin{align}\label{DeltaMF}
		u'(x) = \Ebb \left[ \Phi(X_T^x) \left( \int_0^T a(s) \partial_x X_s^x + \partial_x b(t,y,\varrho_t^{x}(y))\vert_{y=B_t^x} \int_0^s a(u) du dB_s \right) \right],
	\end{align}
	where $a: \Rbb  \to \Rbb$ is any bounded, measurable function such that
	\begin{align*}
		\int_0^T a(s) ds = 1.
	\end{align*}
\end{theorem}
\begin{proof}
	Due to \Cref{lem:RegRepDelta2} we already know that in the case $\Phi \in \Ccal^{1,1}_b(\Rbb)$ the functional $\EW{\Phi(X_T^x)}$ is continuously differentiable and analogously to \cite[Theorem 4.2]{Bauer_StrongSolutionsOfMFSDEs} it can be shown that representation \eqref{DeltaMF} holds. Now, using mollification we can approximate $\Phi\in L^{2p}(\Rbb; \omega_T)$ by a sequence of smooth functionals $\lbrace \Phi_n \rbrace_{n\geq 1} \subset C_0^{\infty}(\Rbb)$ such that $\Phi_n \to \Phi$ in $L^{2p}(\Rbb; \omega_T)$ as $n\to\infty$. We define 
	\begin{align*}
		u_n(x) &:= \Ebb \left[ \Phi_n(X_T^x) \right]\quad \text{and} \\
		\overline{u}(x) &:= \Ebb \left[ \Phi(X_T^x) \left( \int_0^T a(s) \partial_x X_s^x + \partial_x b(t,y,\varrho_t^{x}(y))\vert_{y=B_t^x} \int_0^s a(u) du dB_s \right) \right].
	\end{align*}		
	Note first that $\overline{u}$ is well-defined. Indeed, due to \eqref{eq:boundDerivative}, \Cref{lem:RegBoundsSolution}, and \eqref{eq:boundPartialXb} we get
	\begin{align}\label{eq:RegWellDefined}
		|\overline{u}(x)| &\leq \Ebb \left[ \Phi(X_T^x)^2 \right]^{\frac{1}{2}} \Ebb \left[ \left( \int_0^T a(s) \partial_x X_s^x + \partial_x b(t,y,\varrho_t^{x}(y))\vert_{y=B_t^x} \int_0^s a(u) du dB_s \right)^2 \right]^{\frac{1}{2}} \notag \\
		&\lesssim \Ebb \left[ \Phi(B_T^x)^2 \Ecal\left( \int_0^T b(u,B_u^x,\rho_u^x)dB_u \right) \right]^{\frac{1}{2}} \sup_{s\in[0,T]} \Ebb \left[ \left( \partial_x X_s^x \right)^2 \right]^{\frac{1}{2}} \\
		&\lesssim \Ebb \left[ \left| \Phi(B_T^x) \right|^{2p} \right]^{\frac{1}{2p}} < \infty. \notag
	\end{align}
	Due to \Cref{lem:RegRepDelta2}, $u_n$ is continuously differentiable for all $n\geq 1$. Thus it remains to show that $u_n'(x)$ converges to $\overline{u}(x)$ compactly in $x$ as $n\to\infty$, where denotes $u_n'$ the first derivative of $u_n$ with respect to $x$. Exactly in the same way as in equation \eqref{eq:RegWellDefined} we can find for any compact subset $K\subset \Rbb$ a constant $C$ such that for every $x\in K$
	\begin{align*}
		|u'(x) - \overline{u}(x)| &\leq C \Ebb \left[ \left| \Phi_n(B_T^x) - \Phi(B_T^x)\right|^{2p}\right]^{\frac{1}{2p}} \\
		&= C \left( \int_{\Rbb} \frac{1}{\sqrt{2 \pi T}} \left| \Phi_n(y) - \Phi(y)\right|^{2p} e^{-\frac{(y-x)^2}{2T}} dy \right)^{\frac{1}{2p}} \\
		&\leq C \left( \frac{e^{\frac{x^2}{2T}}}{\sqrt{2 \pi T}} \int_{\Rbb}  \left| \Phi_n(y) - \Phi(y)\right|^{2p} e^{-\frac{y^2}{4T}} dy \right)^{\frac{1}{2p}} \\
		&= C \left( \frac{e^{\frac{x^2}{2T}}}{\sqrt{2 \pi T}}\right)^{\frac{1}{2p}} \left\Vert \Phi_n - \Phi \right\Vert_{L^{2p}(\Rbb;\omega_T)}, 
	\end{align*}
	where we have used $e^{-\frac{(y-x)^2}{2t}} = e^{-\frac{y^2}{4t}} e^{-\frac{(y-2x)^2}{4t}} e^{\frac{x^2}{2t}} \leq e^{-\frac{y^2}{4t}} e^{\frac{x^2}{2t}}$. Consequently
	\begin{align*}
		\lim_{n\to\infty} \sup_{x\in K} |u_n'(x) - \overline{u}(x)| = 0.
	\end{align*}
	 Thus $u'= \overline{u}$ and $u$ is continuously differentiable.
\end{proof}

%Due to the proof of \Cref{thm:RegMainTheoremDelta} we immediately get in the multi-dimensional case, $d\geq 1$, that the Bismut-Elworthy-Li formula holds in the weak sense for the class of coefficients considered in \Cref{thm:RegMainTheoremDelta}.
%
%\begin{corollary}
%	Let $b$ be a measurable bounded function, $\varphi$ be measurable and of at most linear growth \eqref{linearGrowthMF}, and let $b, \varphi \in \Lcal([0,T] \times \Rbb^d \times \Rbb^d)$. Further, let $(X_t^x)_{t\in [0,T]}$ be the unique strong solution of the multi-dimensional mean-field SDE \eqref{lebesgueMFSDE} and $\Phi \in L^{2p}(\Rbb^d;\omega_T)$ with $p:= \frac{1+\varepsilon}{\varepsilon}$, $\varepsilon>0$ sufficiently small with regard to \Cref{lem:RegBoundsSolution}, and $\omega_T(y) := \exp \left\lbrace - \frac{\Vert y \Vert^2}{4T} \right\rbrace$. Then, the expectation functional $ \Ebb\left[ \Phi(X_t^x) \right]$ is Sobolev differentiable in the initial condition and the derivative $\partial_x \Ebb\left[ \Phi(X_T^x) \right]$ admits for almost all $x\in K$, where $K \subset \Rbb^d$ is a compact subset, the representation \eqref{DeltaMF}.
%\end{corollary}

%Appendix____________________________________________________________
\appendix
\section{Technical Results}
The first corollary is due to \cite[Proposition 3.8]{Bauer_StrongSolutionsOfMFSDEs}.

\begin{corollary}\label{cor:L2Convergence}
	Consider $d=1$. Let $(b\diamond \varphi)$ admit a decomposition \eqref{eq:RegFormDrift} and $b$ be Lipschitz continuous in the third variable \eqref{lipschitzThirdMF}. Further let $(X_t^{x})_{t\in[0,T]}$ be the unique strong solution of mean-field SDE \eqref{lebesgueMFSDE} and $\lbrace (X_t^{n,x})_{t\in[0,T]} \rbrace_{n\geq 1}$ be the unique strong solutions of \eqref{eq:RegApproximatingMFSDEExp}. Then, for all $t\in [0,T]$ and every $x \in \Rbb$
	\begin{align*}
		\EpE{X_t^{n,x} - X_t^x}{2} \xrightarrow[n\to\infty]{} 0.
	\end{align*}
\end{corollary}

The upcoming lemma is an extension of \cite[Lemma A.4]{Bauer_StrongSolutionsOfMFSDEs} to multi dimensions.

\begin{lemma}\label{lem:RegBoundsSolution}
	Let $b,\varphi:[0,T] \times \Rbb^d \times \Rbb^d \to \Rbb^d$ be two measurable functions satisfying the linear growth condition \eqref{linearGrowthMF}. Furthermore, let $(\Omega, \Fcal, \Fbb, \Pbb, B, X^x)$ be a weak solution of mean-field SDE \eqref{lebesgueMFSDE}. Then,
	\begin{align*}%\label{eq:RegBoundDriftSolution}
		\left\Vert b\left(t, X_t^x, \int_{\Rbb^d} \varphi(t,X_t^x,z) \Pbb_{X_t^x}(dz)\right) \right\Vert \leq C \left(1+\Vert x \Vert + \sup_{s\in[0,T]} \Vert B_s \Vert \right)
	\end{align*}
	for some constant $C>0$. Consequently, for any compact set $K\subset \Rbb^d$, and $1\leq p < \infty$, there exist $\varepsilon >0$ and a constant $C>0$ such that the following boundaries hold:
	\begin{align*}
		\sup_{x\in K} &~\Ebb \left[\sup_{t\in[0,T]} \left\Vert b\left(t,X_t^x,\int_{\Rbb^d} \varphi(t,X_t^x,z) \Pbb_{X_t^x}(dz)\right)\right\Vert^p\right] < \infty\\
		\sup_{t\in[0,T]} \sup_{x\in K} &~\Ebb \left[\Vert X_t^x\Vert^p\right] \leq C(1+\Vert x\Vert^p) < \infty \\
		\sup_{x\in K} &~\Ebb \left[ \Ecal \left( \int_0^T b\left(u,B_u^x, \int_{\Rbb^d} \varphi(u,B_u^x,z) \Pbb_{X_u^x}(dz)\right) dB_u \right)^{1+\varepsilon} \right] < \infty
	\end{align*}
\end{lemma}

In the following results which are due to \cite[Lemma A.5, Lemma A.6 \& Lemma A.7]{Bauer_StrongSolutionsOfMFSDEs} we use the notation $\bfr_n(t,y) = b_n\left( t,y, \int_\Rbb \varphi(t,y,z) \Pbb_{X_t^{n,x}}(dz) \right)$.

\begin{corollary}\label{cor:RegBoundLocalTime}
	Consider $d=1$. Suppose $(b \diamond \varphi)$ admits a decomposition \eqref{eq:RegFormDrift} and that $b$ and $\varphi$ are Lipschitz continuous in the third variable \eqref{lipschitzThirdMF}. Let $(X_t^x)_{t\in [0,T]}$ be the unique strong solution of mean-field SDE \eqref{lebesgueMFSDE}. Moreover, $\lbrace b_n \rbrace_{n\geq 1}$ is the approximating sequence of $b$ as defined in \eqref{eq:RegApproximatingDrift} and $(X_t^{n,x})_{t\in[0,T]}$, $n\geq 1$, the corresponding unique strong solutions of \eqref{eq:RegApproximatingMFSDEExp}. Then, for all $\lambda \in \Rbb$ and any compact subset $K \subset \Rbb$,
	\begin{align*}
		\sup_{n\geq 0} \sup_{s,t\in[0,T]} \sup_{x\in K} \Ebb \left[ \exp \left\lbrace - \lambda \int_s^t \int_{\Rbb} \bfr_n\left(s,y \right) L^{B^x}(ds,dy) \right\rbrace \right] < \infty.
	\end{align*}
\end{corollary}

\begin{corollary}\label{cor:RegUniformConvergenceEcal}
	Consider $d=1$. Suppose $(b \diamond \varphi)$ admits a decomposition \eqref{eq:RegFormDrift} and that $b$ and $\varphi$ are Lipschitz continuous in the third variable \eqref{lipschitzThirdMF}. Let $(X_t^x)_{t\in[0,T]}$ be the unique strong solution of mean-field SDE \eqref{lebesgueMFSDE}. Furthermore, $\lbrace b_n \rbrace_{n\geq 1}$ is the approximating sequence of $b$ as defined in \eqref{eq:RegApproximatingDrift} and $(X_t^{n,x})_{t\in[0,T]}$, $n\geq 1$, the corresponding unique strong solutions of \eqref{eq:RegApproximatingMFSDEExp}. Then for any compact subset $K \subset \Rbb$ and $q:= \frac{2(1+\varepsilon)}{2+\varepsilon}$, $\varepsilon>0$ sufficiently small with regard to \Cref{lem:RegBoundsSolution},
	\begin{align*}
		\sup_{x\in K} \Ebb \left[ \left| \Ecal \left( \int_0^T \bfr_n(t,B_t^x) dB_t \right) - \Ecal \left( \int_0^T \bfr(t,B_t^x) dB_t \right) \right|^q \right]^{\frac{1}{q}} \xrightarrow[n\to\infty]{} 0.
	\end{align*}
\end{corollary}

\begin{remark}\label{rem:convergenceRho}
	Note that due to \Cref{cor:RegUniformConvergenceEcal} it is readily seen that for any $\psi \in \Lip(\Rbb)$
	\begin{align*}
		\EW{\psi(X_t^{n,x})} \xrightarrow[n\to\infty]{} \EW{\psi(X_t^x)}
	\end{align*}
	uniformly in $t\in[0,T]$ and $x\in K$, where $K\subset \Rbb$ is a compact subset.
\end{remark}

\begin{corollary}\label{cor:RegUniformConvergenceLocalTime}
	Consider $d=1$. Suppose $(b \diamond \varphi)$ admits a decomposition \eqref{eq:RegFormDrift} and that $b$ and $\varphi$ are Lipschitz continuous in the third variable \eqref{lipschitzThirdMF}. Let $(X_t^x)_{t\in[0,T]}$ be the unique strong solution of mean-field SDE \eqref{lebesgueMFSDE}. Furthermore, $\lbrace b_n \rbrace_{n\geq 1}$ is the approximating sequence of $b$ as defined in \eqref{eq:RegApproximatingDrift} and $(X_t^{n,x})_{t\in[0,T]}$, $n\geq 1$, the corresponding unique strong solutions of \eqref{eq:RegApproximatingMFSDEExp}. Then for any compact subset $K \subset \Rbb$, $0 \leq s \leq t \leq T$ and $p\geq 1$,	
	\begin{scriptsize}
	\begin{align*}
		\sup_{x\in K} \Ebb \left[ \left| \exp \left\lbrace - \int_s^t \int_{\Rbb} \bfr_n(u,y) L^{B^x}(du,dy) \right\rbrace - \exp \left\lbrace - \int_s^t \int_{\Rbb} \bfr(u,y) L^{B^x}(du,dy) \right\rbrace \right|^p \right]^{\frac{1}{p}} \xrightarrow[n \to \infty]{} 0.
	\end{align*}
	\end{scriptsize}
\end{corollary}

\section{Skorokhod's representation theorem}

The following result is a version of Skorokhod's representation theorem and is due to \cite[Ch. 1 Sec.6]{skorokhod1982studies}.

\begin{theorem}\label{thm:SkorohodRepresentationTheorem}
	Let $\lbrace (\xi_t^n)_{t\in[0,T]} \rbrace_{n \geq 1}$ be a sequence of $\Rbb^d$-valued stochastic processes defined on probability spaces $(\Omega^n, \Fcal^n, \Pbb^n)$, respectively, which are stochastically continuous from the right and fulfill for every $\varepsilon >0$
	\begin{align*}
		&\lim_{C \to \infty} \lim_{n \to \infty} \sup_{t\in [0,T]} \Pbb^n( \Vert \xi_t^n \Vert > C) = 0, \text{ and } \\
		&\lim_{h\to 0} \lim_{n\to \infty} \sup_{\vert t-s \vert \leq h} \Pbb^n(\Vert \xi_t^n - \xi_s^n \Vert > \varepsilon) = 0.
	\end{align*}
	Then, there exists a subsequence $\lbrace n_k\rbrace_{k\geq 1} \subset \Nbb$ and a sequence of $\Rbb^d$-valued stochastic processes $\lbrace (X_t^k)_{t\in[0,T]} \rbrace_{k\geq 0}$ on a common probability space $(\Omega, \Fcal, \Pbb)$ such that
	\begin{enumerate}[(i)]
		\item for all $k\geq 1$, finite dimensional distributions of the processes $X^k$ and $\xi^{n_k}$ coincide, and 
		\item $X_t^k$ converges in probability to $X_t^0$ for every $t\in [0,T]$.
	\end{enumerate}
\end{theorem}

\begin{remark}\label{rem:finiteDistribution}
	Note that we say that finite dimensional distributions of two processes $X$ and $\xi$, defined on $(\Omega, \Fcal, \Pbb)$ and $(\widetilde{\Omega}, \widetilde{\Fcal}, \widetilde{\Pbb})$, respectively, coincide, if for every finite sequence of time points $\lbrace t_k \rbrace_{1 \leq k \leq N} \subset [0,T]$, $1\leq N < \infty$, we have that
	\begin{align*}
		\Pbb_{(X_{t_1}, \dots X_{t_N})} = \widetilde{\Pbb}_{(\xi_{t_1}, \dots \xi_{t_N})}.
	\end{align*}
\end{remark}

%Bibliography__________________________________________________________
\begin{footnotesize}
	\bibliography{literatureTH}
	\bibliographystyle{abbrv}
\end{footnotesize}
\bigskip
\rule{\textwidth}{1pt}

\end{document}